\documentclass[12pt]{amsart}
\usepackage{amsmath,amssymb,latexsym, amsthm, amscd, mathrsfs}
\usepackage[all]{xy}

\theoremstyle{definition}
\newtheorem{rem}[subsection]{Remark}%[section]
\theoremstyle{plain}
\newtheorem{prop}[subsection]{Proposition}
\newtheorem{thm}[subsection]{Theorem}
\newtheorem{lem}[subsection]{Lemma}
\newtheorem{cor}[subsection]{Corollary}

\newcommand{\mbf}{\mathbf}
\newcommand{\mbb}{\mathbb}
\newcommand{\mrm}{\mathrm}

\newcommand{\B}{\mbf B}
\newcommand{\D}{\mathcal D}

\newcommand{\G}{\mbf G}

\newcommand{\IC}{\mrm{IC}}
\newcommand{\K}{\mathcal K}
\newcommand{\Q}{\bar{\mbb Q}}

\newcommand{\T}{\Theta_{\Omega}}
\newcommand{\U}{\mbf U}
\newcommand{\X}{\mbf X}
\newcommand{\Y}{\mbf Y}

\newcommand{\II}{\mathscr I}
\newcommand{\EE}{\mathscr E}
\newcommand{\FF}{\mathscr F}
\newcommand{\DD}{\mathscr D}
\newcommand{\QQ}{\mathscr Q}
\newcommand{\KK}{\mathscr K}
\newcommand{\BB}{\mathscr B}

\title[Geometric  $(\overset{.}{\U}, \overset{.}{\B})$, II] {On  geometric realizations of  quantum modified algebras and their canonical bases, II}

\author{Yiqiang Li}
\address{Department of Mathematics\\ 
Virginia Polytechnic Institute and State University\\
460 McBryde Hall\\Blacksburg, VA, 24061}
\email{yqli@math.vt.edu}
\date{\today}

\begin{document}

\begin{abstract}
We prove part of the conjectures in ~\cite{Li10b}. 
We  also relate   the construction of quantum modified algebras in ~\cite{Li10b} 
with  the functorial construction in ~\cite{Zheng08}.
\end{abstract}

\maketitle

\tableofcontents

\section{Introduction}

In ~\cite{Li10b}, we propose a geometric construction of quantum modified algebras  $\dot{\U}$ (or rather their quotients) and their canonical bases. 
The construction involves certain localized equivariant derived categories of double framed representation varieties associated to a quiver.
The convolution product is defined by using the left adjoints of the localization functors,  the general direct image functors with compact support 
and the general  inverse image functors. 

In this paper, we show that the complexes $\II_{\mu}$, $\EE^{(n)}_{\mu, \mu-n\alpha_i}$ and 
$\FF^{(n)}_{\mu,\mu+n\alpha_i}$ defined in ~\cite{Li10b} satisfy the defining relations of the quantum modified algebras. 
We then show that the monomials formed  by these complexes are bounded, which is not clear from the definition. 
Finally, we show that the construction in ~\cite{Li10b} and the functorial construction in ~\cite{Zheng08} are compatible, 
in the sense that there is a functor
from the category of complexes in ~\cite{Li10b} to the category of functors in ~\cite{Zheng08} respecting the convolution products.
If the conjectures in ~\cite{Li10b} hold, this functor  gives rise to the positivity property of the structure constants of the action of the canonical basis elements of 
$\dot{\U}$  on the canonical basis elements of the tensor product of the irreducible integrable highest weight representations of $\dot{\U}$.

\subsection{Notation}

\label{notations}

Let  $\Gamma=(I, H, ', '', \bar\empty \; )$ be   a loop-free graph. It corresponds to a symmetric Cartan datum. 
Fix a root datum $(\X, \Y, (,))$ of this Cartan datum and the roots   $\{\alpha_i|i\in I\}$ in $\X$  and the coroots $\{ \check \alpha_i| i\in I\}$ in $\Y$.
It satisfies that $(\check \alpha_i, \alpha_i)=2$ and $(\check \alpha_i, \alpha_j) = - \#\{ h\in H| h'=i, h''=j\}$ for any $i\neq j \in I$.
The set $\X^+$ of dominant integral weights are the collections of elements $\lambda\in \X$ such that $(\check \alpha_i,\lambda)\in \mbb N$ for any $i\in I$.

Throughout this paper,  we fix  an algebraically closed field $k$ of characteristic $p$,
a dominant weight $\lambda$, an element  $d=\sum d_i i\in \mbb N[I]$,
 and an $I$-graded vector space $D$ over $k$ such that $(\check \alpha_i, \lambda)=d_i =\dim D_i$ for any $i\in I$. 
 
 Let $l$ be a prime number different from $p$ and $\Q_l$ an algebraic closure of the field of $l$-adic numbers. 
 We shall refer to  ~\cite{BBD82}, ~\cite{FK88},  ~\cite{KW01}, ~\cite{BL94}, ~\cite{LMB00}, ~\cite{KS90}, ~\cite{LO08a}-\cite{LO09}, ~\cite{Schn08} and ~\cite{WW09} for 
 the definitions of 
 the derived category
 $\D(X)$ of complexes of $\Q_l$-constructible  sheaves on the variety $X$ and  the equivariant derived category
 $\D_G(X)$ of $X$ if the linear algebraic group $G$ acts on $X$. 
 The full subcategories of bounded and bounded below complexes will be denoted by the same notation with a superscript $b$ and $-$, respectively.
 
 If two complexes $K_1$ and $K_2$ are isomorphic, we simply write $K_1=K_2$.
 
 We shall use the notations in ~\cite{Li10b}
 for Grothendieck's six operations. 
 In particular, we write $f_*$, $f^*$, $f_!$ and $f^!$, respectively,  for the   functors $Rf_*$,
 $Lf^*$, $Rf_!$ and $Rf^!$ between $\D(X)$ and $\D(Y)$ if $f: Y\to X$ is a morphism of varieties. 
More generally,  if  the linear algebraic groups $G$ and $G'=G\times G_1$  act on $X$ and $Y$, respectively, such that 
 $f:Y\to X$ is compatible with the group actions, i.e., $f((g, g_1) .y) = g. f(y)$, for any $(g, g_1)\in G'$ and $y\in Y$,  the operators $Rf_*$,  $Rf^*$, $Lf_!$ and $Rf^!$ between
 $\D_G(X)$ and $\D_{G'}(Y)$ in ~\cite{LO08b} will be denoted by $Qf_*$,  $Qf^*$, $Qf_!$ and $Qf^!$, respectively.
 If $G_1$ is trivial, we use  the notations $f_*$, $f^*$, $f_!$ and $f^!$, instead of $Qf_*$,  $Qf^*$, $Qf_!$ and $Qf^!$.
 
 The constant sheaf on $X$ will be denoted by $\Q_{l, X}$, whether or not $G$-equivariant.

\subsection{Quiver variety}
\label{framed}

Fix an orientation  $\Omega$ of $\Gamma$, i.e., $\Omega$ is a subset of $H$ such that $\Omega\sqcup \bar \Omega =H$.
To any $I$-graded vector space $V$ over the field $k$, attached the framed representation variety  of the quiver $(\Gamma, \Omega)$:
\[
\mbf E_{\Omega}(D, V) = \oplus_{h\in \Omega} \mrm{Hom} ( V_{h'} , V_{h''}) \oplus \oplus_{i\in I} \mrm{Hom} (V_i, D_i).
\] 
Elements in $\mbf E_{\Omega}(D, V)$ will be denoted by $X=(x,q)$ where $x$  and $q$ are in the first  and  second  components, respectively.

Let $G_V=\prod_{i\in I} \mrm{GL}(V_i)$ be the product of general linear group $\mrm{GL}(V_i)$.
The group $G_D\times G_V$ acts on $\mbf E_{\Omega}(D, V)$ by conjugation:
\[
(f, g) .(x, q)  =(x', q'), \quad \mbox{where} \; x'_{h} = g_{h''} x_h g_{h'}^{-1}, \;  q'_i= f_i q_i g_i^{-1},  \quad  \forall h\in \Omega, i\in I,
\]
for any $(f, g)\in G_D\times G_V$,   and $(x, q) \in \mbf E_{\Omega}(D, V)$.
To each $i\in I$, we set
\[
X(i)=q_i+\sum_{h\in \Omega: h'=i} x_h : V_i \to D_i\oplus \bigoplus_{h\in \Omega: h'=i} V_{h''} .
\]
To any pair $( V^1, V^2)$ of $I$-graded vector spaces, we set
\[
\mbf E_{\Omega}\equiv  \mbf E_{\Omega}(D, V^1, V^2) = \mbf E_{\Omega}(D,V^1) \oplus  \mbf E_{\Omega}(D, V^2).
\]
The group 
$\G=G_D\times G_{V^1} \times G_{V^2}$
acts on $\mbf E_{\Omega}$ 
by
$(f, g^1, g^2). (X^1, X^2)=( (f, g^1). X^1, (f, g^2) . X^2)$, for any $ (f, g^1, g^2)\in \G, (X^1, X^2) \in \mbf E_{\Omega}$.

Similarly, to a triple  $(V^1, V^2, V^3)$ of $I$-graded vector spaces, we set
\[
\mbf E_{\Omega}(D, V^1, V^2, V^3)= \mbf E_{\Omega}(D, V^1) 
\oplus  \mbf E_{\Omega}(D, V^2) \oplus \mbf E_{\Omega}(D, V^3).
\]
The group $\mbf H=G_D\times G_{V^1} \times G_{V^2}\times G_{V^3}$ acts on 
$\mbf E_{\Omega}(D, V^1, V^2, V^3)$ by 
\[
(f, g^1, g^2, g^3). (X^1, X^2, X^3)=( (f, g^1). X^1, (f, g^2) . X^2, (f, g^3).X^3), 
\]
for any 
$(f, g^1, g^2, g^3)\in \mbf H, (X^1, X^2, X^3) \in \mbf E_{\Omega}(D, V^1, V^2, V^3)$.

\subsection{Fourier-Deligne transform}
\label{Fourier}

Let $\Omega'$ be another orientation of the graph $\Gamma$.  
The various varieties defined in ~\ref{framed} can be defined with respect to $\Omega'$ and $\Omega\cup \Omega'$. 
Define a pairing $u_i: \mbf E_{\Omega\cup \Omega'}(D, V^i)\to k$ by 
$u_i(X^i) =\sum_{h\in \Omega\backslash \Omega'} \mbox{tr}(x^i_hx^i_{\bar h})$ for any $X^i\in \mbf E_{\Omega\cup\Omega'}(D, V^i)$
where $\mbox{tr}(-)$ is the trace of the endomorphism in the parenthesis.
Fix  a non-trivial character $\chi$ from the field $ \mbb F_p$ of $p$ elements to  
$ \Q_l^*:=\Q_l\backslash \{0\}$.
Denote by $\mathcal L_{\chi}$  the local system on $k$ corresponding to $\chi$.
Let 
\begin{equation}
\label{L}
\mathcal L_i =u_i^* \mathcal L_{\chi}.
\end{equation}
Let $u_{ij}: \mbf E_{\Omega \cup \Omega'}(D, V^i, V^j)\to k$  be the pairing defined by $u_{ij} (X^i, X^j) =-u_i(X^i) + u_j(X^j)$ 
for any $(X^i, X^j)\in \mbf E_{\Omega \cup \Omega'}(D, V^i, V^j)$. We set
\begin{equation}
\label{L2}
\mathcal L_{ij}=u_{ij}^* \mathcal L_{\chi}.
\end{equation}
Consider the diagram 
\[
\begin{CD}
\mbf E_{\Omega}(D, V^i, V^j) @<m_{ij} << \mbf E_{\Omega\cup \Omega'} (D, V^i, V^j) @>m_{ij}'>> \mbf E_{\Omega'}(D, V^i, V^j),
\end{CD}
\]
where the morphisms are obvious projections.
The Fourier-Deligne transform
\begin{equation}
\label{Fourier-functor}
\Phi_{\Omega}^{\Omega'}: \D^b_{\G}(\mbf E_{\Omega}(D, V^i, V^j) )\to \D^b_{\G}(\mbf E_{\Omega'}(D, V^i, V^j))
\end{equation}
is defined to be $\Phi_{\Omega}^{\Omega'}(K) =m'_{ij!} (m_{ij}^*(K) \otimes\mathcal L_{ij})[r_{ij}]$, where $r_{ij}$ is the rank of the vector bundle 
$\mbf E_{\Omega}\to \mbf E_{\Omega\cap \Omega'}$ and $\D^b_{\G}(\mbf E_{\Omega})$ is the $\G$-equivariant derived category of $\mbf E_{\Omega}$ (see ~\cite{Li10b}, ~\cite{BL94}). Note that $r_{ij}=\sum_{h\in \Omega\backslash \Omega'} \dim V^i_{h'} \dim V^i_{h''} + \dim V^j_{h'} \dim V^j_{h''}$.

\subsection{Localization}
\label{general-localization}
To each $i\in I$, we fix an orientation $\Omega_i$ of the graph $\Gamma$ such that $i$ is a $source$, 
i.e., if $i\in \{ h', h''\}$ then $h'=i$ for any $h\in \Omega_i$.
Let $F_i$ be the $\G$-invariant closed subvariety of $\mbf E_{\Omega_i}$ consisting of all elements $(X^1, X^2)$ such that either
$X^1(i)$ or $X^2(i)$ is not injective.  
Let $U_i$ be its  $\G$-invariant complement. 

The full subcategory $\mathcal N_i$  of the  category $\D^b_{\G}(\mbf E_{\Omega})$ 
has  objects $K\in \D^b_{\G}(\mbf E_{\Omega})$ satisfying that the support of the complex
$\Phi_{\Omega}^{\Omega_i} (K) $ is contained in the subvariety $F_i$. 
Let $\mathcal N$ be the thick subcategory of  $\D^b_{\G}(\mbf E_{\Omega})$ generated by $\mathcal N_i$ for all $i\in I$.
We define 
\[
\mathscr D^b_{\G} (\mbf E_{\Omega}) = \D^b_{\G}(\mbf E_{\Omega})/ \mathcal N
\]
to be the localization of $\D^b_{\G} (\mbf E_{\Omega})$ with respect to the thick subcategory $\mathcal N$ and 
\[
Q: \D^b_{\G}(\mbf E_{\Omega}) \to  \mathscr D^b_{\G}(\mbf E_{\Omega}) 
\]
the localization functor. See ~\cite{Verdier76} and ~\cite{Li10b} for the details of the localization. 
The localization functor $Q$ admits a fully faithful  right  and left adjoint $Q_*$ and  $Q_!$, respectively, as was shown in ~\cite{Li10b}.

Let $\Omega'$ be another orientation. Let $\mathcal N_{\Omega'}$ be the thick subcategory defined in the same way as $\mathcal N$.
One has, by definition, $\Phi_{\Omega}^{\Omega'} (\mathcal N) = \mathcal N_{\Omega'}$, where $\Phi_{\Omega}^{\Omega'}$ is defined in (\ref{Fourier-functor}).  
So we have an equivalence, 
induced by $\Phi_{\Omega}^{\Omega'}$,
\begin{align}
\label{Fourier-category}
\Phi_{\Omega}^{\Omega'}: \mathscr D^b_{\G}(\mbf E_{\Omega})  \to \mathscr D^b_{\G}(\mbf E_{\Omega'}).
\end{align}

Similarly, one can define the category $\DD^b_{G_D\times G_V}(\mbf E_{\Omega}(D, V))$,
$\DD^b_{\mbf H} (\mbf E_{\Omega} (D, V^1, V^2, V^3))$
and the equivalence of categories $\Phi_{\Omega}^{\Omega'}: \DD^b_{G_D\times G_V}(\mbf E_{\Omega}(D, V))\to \DD^b_{G_D\times G_V}(\mbf E_{\Omega'}(D, V))$ and
\[
\Phi_{\Omega}^{\Omega'}: \mathscr D^b_{\mbf H}(\mbf E_{\Omega}(D, V^1, V^2, V^3))  \to \mathscr D^b_{\mbf H}(\mbf E_{\Omega'}(D, V^1, V^2, V^3)).
\]

\subsection{Convolution product}
\label{convolution}

Let 
\[
p_{ij} : \mbf E_{\Omega} (D, V^1, V^2, V^3)  \to \mbf E_{\Omega} (D, V^i, V^j), 
\]
denote the projection to the $(i, j)$-components. The groups $\mbf H$ and $\mbf G$  in  section ~\ref{framed}
act on  the spaces $\mbf E_{\Omega}$'s,  and 
it is clear that  $p_{ij}$ is compatible with the group actions.
So we have a morphism of algebraic stacks:
\[
Qp_{ij} : [\mbf H\backslash  \mbf E_{\Omega} (D, V^1, V^2, V^3)]   \to [\G\backslash \mbf E_{\Omega} (D, V^i, V^j)].
\]
From ~\cite{LMB00}, ~\cite{LO08a}-\cite{LO09}, we have the following functors:
\begin{align}
\begin{split}
&Qp_{ij}^*: \D^b_{\G}(\mbf E_{\Omega}(D, V^i, V^j)) \to \D^b_{\mbf H} (\mbf E_{\Omega}(D, V^1, V^2, V^3));\\
&(Qp_{ij})_!: \D^-_{\mbf H}(\mbf E_{\Omega}(D, V^1, V^2, V^3)) \to \D^-_{\G}(\mbf E_{\Omega}(D, V^i, V^j)).
\end{split}
\end{align}
We set
\begin{align}
\label{convolution-functors}
\begin{split}
&P_{ij}^* = Q\circ Qp_{ij}^* \circ Q_! : 
\mathscr D^b_{\G}(\mbf E_{\Omega}(D, V^i, V^j)) \to \mathscr D^b_{\mbf H} (\mbf E_{\Omega}(D, V^1, V^2, V^3));\\
& P_{ij!}= Q \circ (Qp_{ij})_!  \circ Q_!: \mathscr D^-_{\mbf H}(\mbf E_{\Omega}(D, V^1, V^2, V^3)) \to \mathscr D^-_{\G}(\mbf E_{\Omega}(D, V^i, V^j)).
\end{split}
\end{align}

To any objects $K \in \mathscr D^-_{\G}(\mbf E_{\Omega} (D, V^1, V^2))$ and $L \in \mathscr D^-_{\G}(\mbf E_{\Omega}(D, V^2, V^3))$, associated
\begin{equation}
\label{cdot}
K \cdot L= P_{13!} (P^*_{12}(K) \otimes P^*_{23}(L)) \quad \in \mathscr D^-_{\G}(\mbf E_{\Omega}(D, V^1, V^3)).
\end{equation}
If, in addition, $M \in \mathscr D^-_{\G}(\mbf E_{\Omega} (D, V^3, V^4))$, we have $(K\cdot L)\cdot M = K \cdot (L\cdot M)$, by ~\cite[Prop. 4.10]{Li10b}.

Let $\Omega'$ be another orientation of the graph $\Gamma$. We can define a similar convolution product, denoted by 
$\cdot_{\Omega'}$, on $\DD^-_{\G}(\mbf E_{\Omega'}(D, V^i, V^j))$'s.  
The following proposition shows that the convolution products are compatible with the Fourier-Deligne transform.

\begin{prop}
\label{Fourier-convolution}
$\Phi_{\Omega}^{\Omega'} (K\cdot L) = \Phi_{\Omega}^{\Omega'} (K) \cdot_{\Omega'} \Phi_{\Omega}^{\Omega'}(L)$,
for any objects $K \in \mathscr D^-_{\G}(\mbf E_{\Omega} (D, V^1, V^2))$ and $L \in \mathscr D^-_{\G}(\mbf E_{\Omega}(D, V^2, V^3))$.
\end{prop}

\begin{proof}
Due to the fact that $m_{13}^*$ is a fully faithful functor and that the condition to define the thick subcategory $\mathcal N$ on 
$\mbf E_{\Omega}$ and $\mbf E_{\Omega\cup \Omega'}$ are the same, one can deduce that $\iota^* m_{13}^* Q_!=0$. From this fact,   
the functor $\Phi_{\Omega}^{\Omega'}: \DD^-_{\G}(\mbf E_{\Omega} (D, V^1, V^3)) \to \DD^-_{\G}(\mbf E_{\Omega'}(D, V^1, V^3))$ can be rewritten as
\[
\Phi_{\Omega}^{\Omega'} (K) =M'_{13!} (M_{13}^* (K) \otimes \mathcal L_{13})[r_{13}].
\]
where  the notations $\mathcal L_{13}$ and $r_{13}$ are defined in ~\ref{Fourier} and $M'_{13!}$ and $M_{13}^*$ are obtained from $m_{13!}'$ and $m_{13}^*$ as in
(\ref{convolution-functors}).
By definition, we have 
\begin{equation*}
\Phi_{\Omega}^{\Omega'} (K\cdot L) =\Phi_{\Omega}^{\Omega'} (P_{13!} (P_{12}^*(K) \otimes P_{23}^* (L)))
=M'_{13!} (M_{13}^* P_{13!} (P_{12}^*(K) \otimes P_{23}^* (L)) \otimes \mathcal L_{13})[r_{13}].
\end{equation*}
Consider the following cartesian diagram
\[
\begin{CD}
\mbf E_{\Omega \cup \Omega'} (D, V^1, V^3) \times \mbf E_{\Omega} (D, V^2) @>s>> \mbf E_{\Omega\cup \Omega'}(D, V^1, V^3)\\
@VrVV @Vm_{13} VV\\
\mbf E_{\Omega}(D, V^1, V^2, V^3) @>p_{13}>> \mbf E_{\Omega}(D, V^1, V^3).
\end{CD}
\]
By an argument similar to ~\cite[(18)]{Li10b}, we have $M_{13}^* P_{13!} =S_! R^*$. So
\begin{equation}
\label{Fourier-convolution-LHS}
\begin{split}
\Phi_{\Omega}^{\Omega'} (K\cdot L)&=M'_{13!} (S_! R^*(P_{12}^*(K) \otimes P_{23}^* (L)) \otimes \mathcal L_{13})[r_{13}]\\
&=M'_{13!} S_! (R^*P_{12}^*(K) \otimes R^* P_{23}^* (L) \otimes S^* \mathcal L_{13})[r_{13}].
\end{split}
\end{equation}

On the other hand, we have
\begin{equation*}
\begin{split}
\Phi_{\Omega}^{\Omega'}(K) & \cdot_{\Omega'} \Phi_{\Omega}^{\Omega'} (L) 
=P'_{13!} ((P'_{12})^* (\Phi_{\Omega}^{\Omega'}(K))\otimes (P'_{23})^* (\Phi_{\Omega}^{\Omega'}(L)))\\
&=P'_{13!} ( (P'_{12})^* M'_{12!} (M_{12}^* (K) \otimes \mathcal L_{12})[r_{12}] \otimes (P'_{23})^* M'_{23!} (M_{23}^*(L)\otimes \mathcal L_{23})[r_{23}]),
\end{split}
\end{equation*}
where  $P_{ij!}'$ and $(P_{ij}')^*$ are obtained from the projections 
$p_{ij}'$ from  $\mbf E_{\Omega'}(D, V^1, V^2, V^3)$ to  $\mbf E_{\Omega'}(D, V^i, V^j)$, $\mathcal L_{ij}$ and $r_{ij}$ are from ~\ref{Fourier}, and
the functors $M'_{ij!}$ and $M_{ij}^*$ are obtained from the projections $m_{ij}$ in ~\ref{Fourier} as $P_{ij!}$ 
and $P_{ij}^*$ from $p_{ij}$ in (\ref{convolution-functors}).
Consider the following cartesian diagrams
\[
\begin{CD}
\mbf E_{\Omega \cup \Omega'} (D, V^1, V^2) \times \mbf E_{\Omega'}(D, V^3) @>s_1' >> \mbf E_{\Omega'} (D, V^1, V^2, V^3) \quad\\
@Vr'_1 VV @V p'_{12} VV \\
\mbf E_{\Omega \cup \Omega'} (D, V^1, V^2) @>m'_{12} >> \mbf E_{\Omega'} (D, V^1, V^2),
\end{CD}
\]
and
\[
\begin{CD}
\mbf E_{\Omega'}(D, V^1)\times \mbf E_{\Omega\cup \Omega'} (D, V^2, V^3) @>s_2'>> \mbf E_{\Omega'}(D, V^1, V^2, V^3)\\
@Vr'_2VV @Vp'_{23} VV\\
\mbf E_{\Omega\cup \Omega'} (D, V^2, V^3) @>m'_{23}>> \mbf E_{\Omega'}(D, V^2, V^3).
\end{CD}
\]
From these cartesian diagrams and similar to ~\cite[(18)]{Li10b}, we have
$(P_{12}')^* M_{12!}' = S'_{1!} (R'_1)^*$ and 
$(P'_{23})^* M'_{23!} = S'_{2!} (R'_2)^*$.
So 
\begin{equation*}
\begin{split}
\Phi&_{\Omega}^{\Omega'}(K)  \cdot_{\Omega'} \Phi_{\Omega}^{\Omega'} (L) =
P'_{13!} ( S'_{1!} (R'_1)^* (M_{12}^* (K) \otimes \mathcal L_{12}) \otimes S'_{2!} (R'_2)^* (M_{23}^*(L)\otimes \mathcal L_{23}))[r_{12}+r_{23}]\\
&=P'_{13!} ( S'_{1!}  ((R'_1)^* M_{12}^* (K) \otimes  (R'_1)^* \mathcal L_{12})\otimes S'_{2!} ((R'_2)^* M_{23}^*(L)\otimes (R'_2)^*( \mathcal L_{23})))[r_{12}+r_{23}]\\
&=P'_{13!} S'_{1!}  ((R'_1)^* M_{12}^* (K) \otimes  (R'_1)^* \mathcal L_{12} \otimes  (S'_1)^*S'_{2!} ((R'_2)^* M_{23}^*(L)\otimes (R'_2)^*( \mathcal L_{23})))[r_{12}+r_{23}].
\end{split}
\end{equation*}
We form the following cartesian diagram
\[
\begin{CD}
\mbf E_{\Omega \cup \Omega'} (D, V^1, V^2,  V^3) 
\times \mbf E^2_{\Omega\backslash \Omega'} (D, V^2)
 @>t'_2>> \mbf E_{\Omega\cup \Omega'}(D, V^1, V^2) \times \mbf E_{\Omega'}(D,V^3)\\
@Vt'_1VV @Vs'_1VV \\
\mbf E_{\Omega'}(D, V^1)\times \mbf E_{\Omega\cup \Omega'} (D, V^2, V^3) @>s'_2>> \mbf E_{\Omega'}(D,V^1, V^2, V^3).
\end{CD}
\]
This cartesian diagram gives rise to the identity, $(S'_1)^* S'_{2!} = T'_{2!} (T'_1)^*$. So
\begin{equation*}
\begin{split}
&\Phi_{\Omega}^{\Omega'}(K)  \cdot_{\Omega'} \Phi_{\Omega}^{\Omega'} (L) =\\
&=P'_{13!} S'_{1!}  ((R'_1)^* M_{12}^* (K) \otimes  (R'_1)^* \mathcal L_{12} \otimes  T'_{2!} (T'_1)^*( (R'_2)^* M_{23}^*(L)\otimes (R'_2)^*( \mathcal L_{23})))[r_{123}]\\
&=P'_{13!} S'_{1!}  T'_{2!} ( (T'_2)^*(R'_1)^* M_{12}^* K  \otimes  (T'_2)^* (R'_1)^* \mathcal L_{12} \otimes(T'_1)^* (R'_2)^* M_{23}^*L\otimes  (T'_1)^*
(R'_2)^* \mathcal L_{23})[r_{123}]\\
&=P'_{13!} S'_{1!}  T'_{2!} 
( (T'_2)^*(R'_1)^* M_{12}^* K \otimes(T'_1)^* (R'_2)^* M_{23}^*L 
\otimes  (T'_2)^* (R'_1)^* \mathcal L_{12}\otimes  (T'_1)^* (R'_2)^* \mathcal L_{23})[r_{123}].
\end{split}
\end{equation*}
where $r_{123}=r_{12}+r_{23}$. Let $\mbf F_1=\mbf E_{\Omega\cup \Omega'}(D, V^1, V^3)$ and 
\[
t'_3: \mbf F\equiv \mbf E_{\Omega \cup \Omega'} (D, V^1, V^2,  V^3) 
\times \mbf E^2_{\Omega\backslash \Omega'} (D, V^2)
 \to \mbf Z\equiv \mbf F_1\times \mbf E_{\Omega}(D, V^2)\times \mbf E^2_{\Omega\backslash \Omega'}(D, V^2)
\]
be the obvious projection.
Note that in the component $\mbf E_{\Omega \cup \Omega'} (D, V^2)$, there is a copy of $\mbf E_{\Omega\backslash \Omega'}(D, V^2)$, 
denoted by $\mbf E^1_{\Omega}(D, V^2)$. 
Observe that  
\[
p'_{13} s'_1 t'_2=w  t'_3, \quad
m_{12} r'_1 t'_2= y_2 t'_3,\quad \mbox{and} \quad 
m_{23} r'_2 t'_1= y_1 t'_3,
\]
where $w$ is the projection from $\mbf Z$ to $\mbf E_{\Omega'}(D, V^1, V^3)$, and $y_1$ and $y_2$ are the projections 
from $\mbf Z$ to  $\mbf E_{\Omega}(D, V^2, V^3)$ (for $y_1$) and $\mbf E_{\Omega}(D, V^1, V^2)$, respectively. 
 Note that there are two choices for the projections for each $y_i$, but we choose the unique one such that the above identities hold.
 So
\begin{equation*}
\begin{split}
\Phi_{\Omega}^{\Omega'}(K)  \cdot_{\Omega'} \Phi_{\Omega}^{\Omega'} (L) 
&=W_! T'_{3!} ( (T'_3)^* Y_2^*  K\otimes (T'_3)^* Y_1^*L \otimes (T'_2)^* (R'_1)^* \mathcal L_{12} \otimes  (T'_1)^* (R'_2)^* \mathcal L_{23})[r_{123}]\\
&=W_! ( Y_2^* K \otimes Y_1^* L \otimes T'_{3!} ((T'_2)^* (R'_1)^* \mathcal L_{12} \otimes  (T'_1)^* (R'_2)^* \mathcal L_{23}))[r_{123}].
\end{split}
\end{equation*}
Let 
$\mbf F_2=\mbf E_{\Omega \cup \Omega'} (D, V^2)\times \mbf E_{\Omega\backslash \Omega'}(D, V^2)$. Thus, 
$\mbf F=\mbf F_1\times \mbf F_2$.
Each component $\mbf E^i_{\Omega}(D, V^2)$ defines a projection,
$\pi_{i,i+1}:\mbf F_2\to \mbf E_{\Omega\cup \Omega'} (D, V^2)$, for any $ i=1, 2$.
We have 
\[
r_1' t_2' =w_1 (1\times \pi_{23}) 
\quad \mbox{and} \quad 
r_2' t_1' =w_2 (1\times \pi_{12}),
\]
where $1: \mbf F_1\to \mbf F_1$ is the identity map and  $w_i$ is the projection of  $\mbf E_{\Omega\cup \Omega'}(D, V^1, V^2, V^3)$ 
to $\mbf E_{\Omega\cup \Omega'}(D, V^i, V^{i+1})$ for any $i=1, 2$. Hence,
\begin{equation*}
\Phi_{\Omega}^{\Omega'}(K)  \cdot_{\Omega'} \Phi_{\Omega}^{\Omega'} (L)=
W_! ( Y_2^* K \otimes Y_1^* L \otimes T'_{3!}
((1\times \Pi_{23})^* W_1^*  \mathcal L_{12} \otimes  (1\times \Pi_{12})^* W_2^* \mathcal L_{23}))[r_{123}].
\end{equation*}
Observe that  $W_1^* \mathcal L_{12} = P^*_1  \mathcal L_1^* \otimes P_2^*  \mathcal L_2$ 
and $W_2^* \mathcal L_{23}=P_3^* \mathcal L_3 \otimes P_2^* \mathcal L_2^*$
where $\mathcal L_i^*$ is the dual of the local system $\mathcal L_i$ in  ~\ref{Fourier}, 
and  $p_i$ are the projections from $\mbf E_{\Omega\cup\Omega'}(D, V^1, V^2,  V^3)$ to $\mbf E_{\Omega\cup \Omega'}(D, V^i)$. So
\begin{equation*}
\begin{split}
T'_{3!}&((1\times \Pi_{23})^* W_1^*  \mathcal L_{12} \otimes  (1\times \Pi_{12})^* W_2^* \mathcal L_{23})=\\
&
=T'_{3!} (((1\times \Pi_{23})^* (P_1^* \mathcal L_1^* \otimes P_2^* \mathcal L_2)  \otimes (1\times \Pi_{12})^*  P_2^*( \mathcal L^*_2 \otimes P_3^*\mathcal L_3))\\
&=T'_{3!} ( (T'_3)^* P^* \mathcal L_{13} \otimes  (1\times \Pi_{23})^*  P_2^* \mathcal L_2\otimes (1\times \Pi_{12})^*  P_2^*( \mathcal L^*_2 ))\\
&=P^* \mathcal L_{13} \otimes T'_{3!} (1\times \Pi_{23})^*  P_2^* \mathcal L_2\otimes (1\times \Pi_{12})^*  P_2^*( \mathcal L^*_2 ))
\end{split}
\end{equation*}
where $P^*$ comes from the projection $p: \mbf Z \to \mbf E_{\Omega\cup \Omega'} (V^1, V^3)$.
By a similar argument as ~\cite[p. 44]{KW01}, we have 
\[
T'_{3!} (1\times \Pi_{23})^*  P_2^* \mathcal L_2\otimes (1\times \Pi_{12})^*  P_2^*( \mathcal L^*_2 ))=
\Delta_! \bar {\mbb Q}_{l, \mbf Z_1} [-2 r'].
\]
where $\Delta_!$ is from the diagonal map 
$\delta: \mbf Z_1\equiv \mbf E_{\Omega\cup \Omega'} (D, V^1, V^3) \times \mbf E_{\Omega}(D, V^2)  \to \mbf Z$ 
and $r'$ is the rank of $t_3'$, which is equal to $\sum_{h\in \Omega\backslash \Omega'} \dim V^2_{h'}\dim V^2_{h''}$. So 
\[
P^* \mathcal L_{13} \otimes T'_{3!} (1\times \Pi_{23})^*  P_2^* \mathcal L_2\otimes (1\times \Pi_{12})^*  P_2^*( \mathcal L^*_2 ))=
P^* \mathcal L_{13} \otimes\Delta_! \bar {\mbb Q}_{l, \mbf Z_1} [-2 r']=
\Delta_! S^* \mathcal L_{13}[-2r'].
\]
Therefore,
\begin{equation}
\label{Fourier-convolution-RHS}
\begin{split}
\Phi_{\Omega}^{\Omega'}(K)  \cdot_{\Omega'} \Phi_{\Omega}^{\Omega'} (L) &=
W_! ( Y_2^* K \otimes Y_1^* L \otimes\Delta_! S^* \mathcal L_{13}) [r_{123}-2r']\\
&=W_! \Delta_! (\Delta^* Y_2^* K\otimes \Delta^* Y_1^* L \otimes S^* \mathcal L_{13}) [r_{13}]\\
&=M'_{13!} S_! (R^*P_{12}^*(K) \otimes R^* P_{23}^* (L) \otimes S^* \mathcal L_{13})[r_{13}],
\end{split}
\end{equation}
where the last identity comes from the observation that $w\delta=m_{13}' s$, $y_2\delta = p_{12} r$ and $y_1 \delta  =p_{23} r$. 
The proposition follows  from (\ref{Fourier-convolution-LHS}) and (\ref{Fourier-convolution-RHS}).
\end{proof}

\subsection{Generator}
Given any pair $(X^1, X^2 ) \in \mbf E_{\Omega}(D, V^1, V^2)$, we write 
\[
\mbox{``}\; X^1\hookrightarrow X^2\;\mbox{''}
\] 
if there exists an $I$-graded inclusion  $\rho: V^1\to V^2$
such that  $\rho_{h''} x^1_h =x^2_h \rho_{h'}, \; q^1_i=q^2_i \rho_i$, for any  $h$ in $\Omega$ and $i$ in $I$.
We also write ``$\rho: X^1\hookrightarrow X^2$''  for such a $\rho$ and ``$X^1\overset{\rho}{\hookrightarrow} X^2$'' for the triple $(X^1, X^2, \rho)$.
Consider the smooth variety
\begin{equation}
\label{Z}
\mbf Z_{\Omega} \equiv \mbf Z_{\Omega}(D, V^1, V^2)=\{ (X^1, X^2, \rho ) | (X^1, X^2) \in \mbf E_{\Omega}(D, V^1, V^2)\; \mbox{and}\; \rho: X^1 \hookrightarrow X^2\}.
\end{equation}
Then we have a diagram
\begin{equation}
\begin{CD}
\mbf E_{\Omega}(D, V^1) @<\pi_1<< \mbf Z_{\Omega}(D, V^1, V^2) @>\pi_2>> \mbf E_{\Omega}(D, V^2),\\
@|  @V\pi_{12}VV @|\\
\mbf E_{\Omega}(D, V^1)@<p_1<< \mbf E_{\Omega}(D, V^1, V^2) @>p_2>> \mbf E_{\Omega}(D, V^2),
\end{CD}
\end{equation}
 where $\pi_1$ and $p_1$ are projections to the first components,  $\pi_2$ and $p_2$ are the projections to the second components,  
 and $\pi_{12}$ is the projection to $(1, 2)$  components.

Similar to $\mbf Z_{\Omega}(D, V^1, V^2)$, let
$\mbf Z^t_{\Omega}(D, V^1, V^2)=\{(X^1, X^2, \rho) | \rho: X^2\hookrightarrow X^1\}$. 
Let $\pi_{12}$ denote the projection $\mbf Z^t_{\Omega}(D, V^1, V^2) \to \mbf E_{\Omega}(D, V^1, V^2)$.
Note that $\mbf Z_{\Omega}(D, V^1, V^2)\simeq \mbf Z^t_{\Omega}(D, V^2, V^1)$.

 We set  the following complexes in $\mathscr D^b_{\G}( \mbf E_{\Omega})$ with $\mu=\lambda-\nu$ and $n\in \mbb N$:
\begin{align*}
\II_{\mu}  &= Q \left (\pi_{12!} (\bar{\mbb Q}_{l,  \mbf Z_{\Omega}}) \right ),  
&&\text{if  $\dim V^1=\dim V^2=\nu$}; \\
\EE^{(n)}_{\mu, \mu-n \alpha_i}  &=Q \left ( \pi_{12!}  (\bar{\mbb Q}_{l, \mbf Z_{\Omega}}) [e_{\mu,n\alpha_i}]  \right ), 
&&\text{if    $\dim V^1=\nu$ and $ \dim V^2=\nu+ ni$};\\ 
\FF^{(n)}_{\mu, \mu +n \alpha_i} &=  Q\left ( \pi_{12!}  (\bar{\mbb Q}_{l, \mbf Z^t_{\Omega}}) [f_{\mu,n\alpha_i}] \right ),
&&\text{if  $\dim V^1=\nu$ and $\dim V^2=\nu-ni$};
\end{align*}
where 
\begin{equation}
\label{coefficient} 
e_{\mu, n\alpha_i}=n \left (d_i +\sum_{h\in \Omega: h'=i} \nu_{h''} -(\nu_i+n) \right )
\quad \mbox{and}\quad  
f_{\mu, n\alpha_i}=n\left ((\nu_i -n)-\sum_{h\in \Omega: h''=i} \nu_{h'}\right ).
\end{equation}
We set $\II_{\mu}$, $\EE^{(n)}_{\mu,\mu-n\alpha_i}$ and $\FF^{(n)}_{\mu,\mu+n\alpha_i}$ to 
be zero if $\mu\in \mbf X$ can not be written as the form $\mu=\lambda-\nu$ for some $\nu\in \mbb N[I]$.
Note that the complexes $\II_{\mu}$, $\EE^{(n)}_{\mu,\mu-n\alpha_i}$ and $\FF^{(n)}_{\mu,\mu+n\alpha_i}$ are denoted by
$\II_{\mu}^{\Omega}$, $\EE^{(n),\Omega}_{\mu,\mu-n\alpha_i}$ and $\FF^{(n),\Omega}_{\mu,\mu+n\alpha_i}$ in ~\cite{Li10b}, respectively.

By  modifying the proof of ~\cite[Lemma 5.6]{Li10b}, one can prove the following lemma.

\begin{lem}
\label{generator-Fourier}
$\Phi_{\Omega}^{\Omega'} (\II_{\mu}) =\II_{\mu}$, 
$\Phi_{\Omega}^{\Omega'} (\EE^{(n)}_{\mu, \mu-n \alpha_i} ) =\EE^{(n)}_{\mu, \mu-n \alpha_i} $
and 
$\Phi_{\Omega}^{\Omega'}(\FF^{(n)}_{\mu, \mu +n \alpha_i}) =\FF^{(n)}_{\mu, \mu +n \alpha_i}$
where the elements on the right-hand sides are complexes in $\DD^b_{\G}(\mbf E_{\Omega'})$ defined in the similar way as the complexes on the left-hand sides.
\end{lem}

Similar to the functors $P_{ij!}$ and $P_{ij}^*$, we define ($i=1, 2$)
\begin{equation}
\label{Pi}
\begin{split}
P_{i!} = Q\circ Qp_{i!} \circ Q_! , \;  P_i^* = Q\circ Qp_i^* \circ Q_!; \quad 
\Pi_{i!}= Q\circ Q\pi_{i!} \circ Q_! \; \mbox{and} \;
\Pi_i^*= Q\circ Q\pi_i^*\circ Q_!.
\end{split}
\end{equation}

\subsection{Defining relation} 

We shall show that the complexes $\II_{\mu}$, $\EE^{(n)}_{\mu, \mu-n \alpha_i}$ and $\FF^{(n)}_{\mu, \mu +n \alpha_i}$ satisfy the defining relations of $\dot{\U}$.

\begin{lem}
\label{I-I}
$\II_{\mu}\II_{\mu'} =\delta_{\mu, \mu'} \II_{\mu}$ where $\mu'=\lambda -\nu'$ for any $\nu'\in \mbb N[I]$.
\end{lem}

\begin{proof}
Assume that $V^1=V^2=V^3$ has dimension $\nu$. Consider the following cartesian diagram
\[
\begin{CD}
\mbf Z_{\Omega}(D, V^1, V^2) @< << \mbf Z_1 @>s_1>> \mbf Z\\
@V\pi_{12}VV @Vr_1VV @Vs_2VV\\
\mbf E_{\Omega}(D, V^1, V^2) @<p_{12}<< \mbf E_{\Omega}(D, V^1, V^2, V^3) @< r_2<< \mbf Z_2,
\end{CD}
\]
where $\mbf Z_1=\mbf Z_{\Omega}(D, V^1, V^2) \times \mbf E_{\Omega}(D, V^3)$, $\mbf Z_2=\mbf E_{\Omega}(D, V^1)\times \mbf Z_{\Omega}(D, V^2, V^3)$,
\[
\mbf Z=\mbf Z_1\times_{\mbf E_{\Omega} (D, V^1, V^2, V^3)} \mbf Z_2= \{(X^1, X^2, X^3, \rho_1, \rho_2) | 
X^1\overset{\rho_1}{\hookrightarrow} X^2 \overset{\rho_2}{\hookrightarrow} X^3\}.
\]
and  the  morphisms  are the obvious projections. 
The cartesian square on the left gives rise to the identity 
$P_{12}^* \Pi_{12!} (\bar{\mbb Q}_{l,\mbf Z_{\Omega}(D, V^1, V^2)})= R_{1!} (\bar{\mbb Q}_{l, \mbf Z_1})$.
Similarly, $P_{23}^* \Pi_{12!} (\bar{\mbb Q}_{l,\mbf Z_{\Omega}(D, V^2, V^3)})= R_{2!} (\bar{\mbb Q}_{l, \mbf Z_2})$. 
So 
\begin{equation*}
\begin{split}
\II_{\mu} \cdot \II_{\mu} & = P_{13!} (P_{12}^* (\II_{\mu}) \otimes P_{23}^* (\II_{\mu})) 
= P_{13!}(R_{1!} (\Q_{l, \mbf Z_1})\otimes R_{2!} (\Q_{l, \mbf Z_2}))\\
&=P_{13!} R_{1!} (\Q_{l, \mbf Z_1}\otimes R_1^* R_{2!} (\Q_{l, \mbf Z_2}))
=P_{13!} R_{1!} (R_1^* R_{2!} (\Q_{l, \mbf Z_2}) ).
\end{split}
\end{equation*}
The right cartesian square in the above diagram   implies that $R_1^* R_{2!} = S_{1!} S_2^*$.
Thus
\[
\II_{\mu} \cdot \II_{\mu}=P_{13!} R_{1!} (R_1^* R_{2!} (\Q_{l, \mbf Z_2}) ) =
P_{13!} R_{1!}S_{1!} S_2^* (\Q_{l, \mbf Z_2})= P_{13!} R_{1!}S_{1!} (\Q_{l, \mbf Z}).
\]
Consider the following commutative  diagram 
\[
\begin{CD}
\mbf Z @>t>> \mbf Z_{\Omega}(D, V^1, V^3)\\
@Vr_1s_1VV @V\pi_{12}VV \\
\mbf E_{\Omega}(D, V^1, V^2. V^3) @> p_{13}>> \mbf E_{\Omega}(D, V^1, V^3),
\end{CD}
\]
where $t$ sends $X^1\overset{\rho_1}{\hookrightarrow} X^2 \overset{\rho_2}{\hookrightarrow} X^3$ to $X^1\overset{\rho_2 \rho_1}{\hookrightarrow} X^3$. Then we have
\[
\II_{\mu} \cdot \II_{\mu}=P_{13!} R_{1!}S_{1!} (\Q_{l, \mbf Z})= \Pi_{12!} T_! (\Q_{l, \mbf Z}).
\]
Observe that $t$ is the quotient map of $\mbf Z$ by the group  $G_{V^2}$, thus the induced morphism $Qt: [\mbf H\backslash \mbf Z] \to [\G\backslash \mbf Z_{\Omega}(D, V^1, V^3)]$ is an isomorphism.
Hence $T_! (\Q_{l, \mbf Z})=\Q_{l,\mbf Z_{\Omega}(D, V^1, V^3)}$. Therefore,
\[
\II_{\mu} \cdot \II_{\mu}=\Pi_{12!} (\Q_{l,\mbf Z_{\Omega}(D, V^1, V^3)} )= \pi_{12!} (\Q_{l,\mbf Z_{\Omega}(D, V^1, V^3)} )=\II_{\mu}.
\]
It is clear that $\II_{\mu} \cdot \II_{\mu'}=0$ if $\mu\neq \mu'$ from the above argument. The lemma follows.
\end{proof}

\begin{lem} 
For any $\mu$ and $\mu'$, we have  
\begin{alignat*}{3}
 \EE^{(n)}_{\mu,\mu-n\alpha_i} \II_{\mu'}&=\delta_{\mu-n\alpha_i,\mu'} \EE^{(n)}_{\mu, \mu-n\alpha_i},\quad 
& \II_{\mu'} \EE^{(n)}_{\mu, \mu-n\alpha_i} &=\delta_{\mu', \mu} \EE^{(n)}_{\mu, \mu-n\alpha_i}; \\
\FF^{(n)}_{\mu, \mu +n\alpha_i} \II_{\mu'}&=\delta_{\mu+n\alpha_i,\mu'} \FF^{(n)}_{\mu, \mu+n\alpha_i}, \quad
&\II_{\mu'} \FF^{(n)}_{\mu, \mu+n \alpha_i} &= \delta_{\mu', \mu} \FF^{(n)}_{\mu, \mu+n\alpha_i}.
\end{alignat*}
\end{lem}

This lemma  can be  proved in exactly the same way as the proof of Lemma ~\ref{I-I}.

\begin{lem}
\label{E-i-j}
$\EE_{\mu -\alpha_j+\alpha_i,\mu -\alpha_j} \FF_{\mu -\alpha_j,\mu}=  \FF_{\mu +\alpha_i-\alpha_j,\mu +\alpha_i} \EE_{\mu+\alpha_i,\mu} $, for any $i\neq j$.
\end{lem}

\begin{proof}
Fix four $I$-graded vector spaces $V^i$ for $i=1, 2, 3, 4$ such that
\begin{equation}
\label{E-F-i-j}
\dim V^1 =\nu +j-i, \quad \dim V^2 =\nu +j, \quad \dim V^3 = \nu \quad\mbox{and} \quad 
\dim V^4 =\nu-i.
\end{equation}
Let 
\begin{equation*}
\mbf Z_1
=\{ (X^1, X^2, X^3, \rho_1, \rho_2) | (X^1, X^2, X^3)\in \mbf E_{\Omega}(D, V^1, V^2, V^3), X^1 \overset{\rho_1}{\hookrightarrow} X^2 \overset{\rho_2}{\hookleftarrow} X^3\},
\end{equation*}
and $\pi_1$ be the projection from $\mbf Z_1$ to $\mbf E_{\Omega}(D, V^1, V^2, V^3)$. 
Then an argument similar to the proof of Lemma ~\ref{I-I} yields that 
\[
\EE_{\mu -\alpha_j+\alpha_i,\mu -\alpha_j} \FF_{\mu -\alpha_j,\mu}=P_{13!} \Pi_{1!} (\Q_{l,\mbf Z_1}) [m],
\]
where
$m=e_{\mu-\alpha_j+\alpha_i, \alpha_i}+f_{\mu-\alpha_j, \alpha_j}$.
Denote by $\check V$ the $I\backslash \{j\}$-graded vector space obtained from $V$ by deleting the component $V_j$. Let 
$\mbf Z$ be the variety of quadruples  $( X^1, X^3,  \check \rho: \check V^1\hookrightarrow  \check V^3, \sigma_j: V^3_j \hookrightarrow V^1_j)$, where  
$(X^1, X^3)\in \mbf E_{\Omega}(D, V^1, V^3)$, such that all the diagrams  incurred in the quadruples are commutative, i.e.,
\[
x^3_h \check \rho_{ h'}=\check \rho_{ h''} x^1_h, \; \mbox{if}\;  \{h', h''\} \neq j;\;
x_h^1 = \sigma_j x^3_h \check \rho_{ h'},\; \mbox{if} \; h''=j;\; \mbox{and}\; 
\sigma_j x^1_h = x^3_h \check \rho_{h'},\; \mbox{if} \; h'=j.
\]
Define a morphism of varieties
\[
r_{13}: \mbf Z_1\to \mbf Z
\]
by $r_{13}(X^1, X^2, X^3, \rho_1, \rho_2) = (X^1, X^3, \check \rho, \sigma_j)$ where 
$\check \rho_{k} = \rho_{2, k}^{-1} \rho_{1, k}$ for any $k\in I\backslash\{ j\}$ and $\sigma_j= \rho_{1, j}^{-1} \rho_{2,j}$.
Then, we have $p_{13} \pi_2= \pi r_{13}$, where $\pi$ is the projection from $\mbf Z$ to $\mbf E_{\Omega}(D, V^1, V^3)$. Moreover,
we observe that $r_{13}$ is a quotient map of $\mbf Z_1$ by $G_{V^2}$. From these facts, we have
\begin{equation}
\label{E-F-LHS}
\EE_{\mu -\alpha_j+\alpha_i,\mu -\alpha_j} \FF_{\mu -\alpha_j,\mu}=P_{13!} \Pi_{1!} (\Q_{l,\mbf Z_1})[m]
=\Pi_! R_{13!} (\Q_{l,\mbf Z_1}) [m]=\Pi_! (\Q_{l, \mbf Z})[m].
\end{equation}

On the other hand, let 
\begin{equation*}
\mbf Z_2
=\{ (X^1, X^4, X^3, \rho_1, \rho_2) | (X^1, X^4, X^3)\in \mbf E_{\Omega}(D, V^1, V^4, V^3), X^1 \overset{\rho_1}{\hookleftarrow} X^4 \overset{\rho_2}{\hookrightarrow} X^3\}
\end{equation*}
and $\pi_2$ be the projection from $\mbf Z_2$  to $\mbf E_{\Omega}(D, V^1, V^4, V^3)$. Define a morphism $\tilde r_{13}: \mbf Z_2\to \mbf Z$ by 
$\tilde r_{13} (X^1, X^4, X^3, \rho_1, \rho_2)=(X^1, X^3, \check \rho, \sigma_j)$ where 
$\check \rho_{ k} = \rho_{2,k} \rho_{1,k}^{-1}$ for any $k\in I\backslash \{j\}$ and 
$\sigma_j= \rho_{1,j} \rho_{2,j}^{-1}$.
Then we have $\tilde p_{13} \pi_2 = \pi \tilde r_{13}$, where $\tilde p_{13}$ is the projection from $\mbf Z_2$ to $\mbf E_{\Omega}(D, V^1, V^3)$, 
moreover $\tilde r_{13}$ is quotient map of $\mbf Z_2$ by $G_{V^4}$. From these facts, we get
\begin{equation}
\label{E-F-RHS}
 \FF_{\mu +\alpha_i-\alpha_j,\mu +\alpha_i} \EE_{\mu+\alpha_i,\mu} =
 \tilde P_{13!} \Pi_{2!} (\Q_{l,\mbf Z_2})[m']  = \Pi_! \tilde R_{13!} (\Q_{l,\mbf Z_2}) [m']=\Pi_! (\Q_{l, \mbf Z})[m'],
\end{equation}
where $m'=f_{\mu+\alpha_i-\alpha_j,\alpha_j} +e_{\mu+\alpha_i,\alpha_i}$.
By (\ref{E-F-LHS}), (\ref{E-F-RHS}) and the fact that $m=m'$, we have
the lemma.
\end{proof}

\begin{lem}  
\label{E-F-i} 
Let $\nu(i)=d_i +\sum_{h\in H: h'=i} \nu_{h''}$. For any vertex $i\in I$, 
\[
\EE_{\mu, \mu-\alpha_i} \FF_{\mu-\alpha_i, \mu} \oplus \bigoplus_{p=0}^{\nu_i -1} \II_{\mu} [\nu(i) -1-2p] =
\FF_{\mu, \mu+\alpha_i} \EE_{\mu+\alpha_i,\mu} \oplus \bigoplus_{p=0}^{\nu(i) -\nu_i-1}\II_{\mu}[\nu(i)-1-2p].
\]

\end{lem}

\begin{proof}
Due to Lemma ~\ref{generator-Fourier},  we may assume that $i$ is a source in $\Omega$.  Let us fix four $I$-graded vector spaces, $V^a$, for $a=1, 2, 3, 4$,  such that
\begin{equation}
\label{E-i-i}
\dim V^1=\dim V^3=\nu, \quad \dim V^2=\nu+i \quad \mbox{and} \quad \dim V^4=\nu-i.
\end{equation}
Then  we have 
\[
\EE_{\mu,\mu -\alpha_i} \FF_{\mu -\alpha_i,\mu}=P_{13!} \Pi_{1!} (\Q_{l,\mbf Z_1}) [m]
\quad \mbox{and} \quad
\FF_{\mu +\alpha_i-\alpha_j,\mu +\alpha_i} \EE_{\mu+\alpha_i,\mu} =
 \tilde P_{13!} \Pi_{2!} (\Q_{l,\mbf Z_2})[m],
\]
where $m=\nu(i) -1$ and  
the other notations on the right-hand sides are defined  in the proof of Lemma ~\ref{E-i-j}  with the condition (\ref{E-F-i-j}) replaced by 
(\ref{E-i-i}).

Let $\mbf Z_1^s$ be the open subvariety of $\mbf Z_1$ defined by the condition that $X^1(i)$, $X^2(i)$ and $X^3(i)$ are injective.
Similarly, we define the open subvariety $\mbf E_{\Omega}^s$ in $\mbf E_{\Omega}(D, V^1, V^3)$.

Denote by $\check X$  the element obtained from $X\in \mbf E_{\Omega}(D, V)$ by deleting any component $x_h$ such that $ h'=i$.
Let $\check {\mbf Z}_1^s$ be the variety of tuples $( X^1,\check X^2, X^3, \rho_1, \rho_2)$.
Let $\mbf Y_1$ be the variety of  tuples $(\check X^1, \check X^2, \check X^3, \check \rho_1, \check \rho_2, \mathcal V_1,\mathcal  V_2, \mathcal V_3)$, 
where $\mathcal V_1, \mathcal V_2, \mathcal V_3\subseteq D_i \oplus \oplus_{h\in \Omega: h'=i}V_{h''}$, 
such that $\mathcal V_1,\mathcal  \mathcal \mathcal V_3\subseteq \mathcal V_2$,
$\dim \mathcal V_1=\dim \mathcal V_3=\nu_i$ and
$\dim \mathcal V_2=\nu_i+1$.
Similarly, we define the variety $\X_1$ of tuples $(\check X^1, \check X^2, \check X^3, \check \rho_1, \check \rho_2, \mathcal V_1, \mathcal V_3)$ and 
the variety $\mbf W$ of tuples $(\check X^1,\check X^3, \mathcal V_1, \mathcal V_3)$.
Then we have the following cartesian diagram
\[
\begin{CD}
\mbf Z_1^s @>r_1 >>  \check {\mbf Z}_1^s @>r_2>> \mbf E_{\Omega}^s\\
@Vs_1 VV @Vs_2 VV @Vs_3VV\\
\mbf Y_1 @>r_3>> \mbf X_1 @>r_4 >> \mbf W,
\end{CD}
\]
where the $r_a$'s  are the obvious projections, and $s_2$ and $s_3$ are induced from $s_1$, which is defined by
$s_1 ( X^1, X^2, X^3,\rho_1,\rho_2) = (\check X^1, \check X^2, \check X^3, \check \rho_1, \check \rho_2, \mathcal V_1,\mathcal  V_2, \mathcal V_3)$
with 
\[
\mathcal V_2 =\mbox{im} \; (q_i^2+\sum_{h\in \Omega: h'=i} x_h^2),\quad \mbox{and}\quad
\mathcal V_a=\mbox{im} \; (q_i^a+\sum_{h\in \Omega: h'=i} \rho_{a,h''}x_h^a),\quad \forall a=1, 3.
\]
Observe that $s_1$ and $s_2$ are the quotient maps of $\mbf Z_1^s$ and $\check{\mbf Z}^s_1$ by the group 
$G_{V_i^1}\times G_{V_i^2}\times G_{V_i^3}$, respectively, and $s_3$ is the quotient map of $\mbf E^s_{\Omega}$ by  $G_{V_i^1}\times G_{V_i^3}$. Thus we have
\begin{equation}
R_{2!} R_{1!} (\Q_{l,\mbf Z_1^s}) 
=S_3^* R_{4!} R_{3!} (\Q_{l,\mbf Y_1}). 
\end{equation}
Let $\mbf Y_1^c$ be the closed subvariety of $\mbf Y_1$ defined by the condition $\mathcal V_1=\mathcal V_3$ and $\mbf Y_1^o$ be its complement.
Let $i_1: \mbf Y_1^c\to \mbf Y_1$ and $j_1:\mbf  Y_1^o\to \mbf Y_1$ be the inclusions.
Sine $r_3$ is proper and $\mbf Y_1$ is smooth,  the complex $R_{3!} (\Q_{l,\mbf Y_1})$ is semisimple. So we have 
\[
r_{3!} (\Q_{l,\mbf Y_1}) = j_{1!*}^o r^o_{3!}  (\Q_{l, \mbf Y_1^o}) \oplus r_{3!} i_{1!} (\Q_{l,\mbf Y_1^c}),
\]
where $j^o_1$ and $r_3^o$ are the morphisms $\mbf Y_1^o\overset{r_3^o} {\to} \mbf X_1^o  \overset{j_1^o}{\to} \mbf X_1$ with $\mbf X_1^o$ the image of $\mbf Y_1^o$
under $r_3$.
Observe that the morphism $r_3 i_1$ is a projective bundle of relative dimension $\nu(i)-\nu_i-1$. Thus
$r_{3!} i_{1!} (\Q_{l,\mbf Y_1^c})= \oplus_{p=0}^{\nu(i)-\nu_i -1} \Q_{l, \mbf X_1^c}[-2p]$ 
where $\mbf X_1^c$ is the closed subvariety of $\mbf X_1$ defined by the condition 
$\mathcal V_1=\mathcal V_3$.
Then, we have 
\[
R_{3!} (\Q_{l,\mbf Y_1})= J_{!*}^o R^o_{3!}  (\Q_{l, \mbf Y_1^o}) \oplus  \oplus_{p=0}^{\nu(i)-\nu_i -1} \Q_{l, \mbf X_1^c}[-2p].
\]

By combining the above analysis, we see that the restriction of the complex $\EE_{\mu,\mu -\alpha_i} \FF_{\mu -\alpha_i,\mu}$ to $\mbf E_{\Omega}^s$ 
is equal to 
\begin{equation}
\label{E-i-i-LHS}
R_{2!}R_{1!} (\Q_{l,\mbf Z_1^s}) = S_3^* R_{4!} R_{3!} (\Q_{l,\mbf Y_1})=
S_3^* R_{4!} J_{1!*}^o R^o_{3!}  (\Q_{l, \mbf Y_1^o}) \oplus  \oplus_{p=0}^{\nu(i)-\nu_i -1} S_3^* R_{4!}\Q_{l, \mbf X_1^c}[m-2p].
\end{equation}

On the other hand, we may define the  open subvarieties $\mbf Z_2^s$ of $\mbf Z_2^s$ similar to the subvariety $\mbf Z_1^s$ of $\mbf Z$. 
Then the following varieties $\check{\mbf Z}_2$,
$\mbf Y_2$ , $\mbf Y_2^c$, $\mbf Y_2^o$ and $\mbf X_2$ in the diagram below are defined in a way similar to the varieties having  subscript $1$:
\[
\begin{CD}
\mbf Z_2^s @>t_1 >>  \check {\mbf Z}_2^s @>t_2>> \mbf E_{\Omega}^s\\
@Vw_1 VV @Vw_2 VV @Vs_3VV\\
\mbf Y_2 @>t_3>> \mbf X_2 @>t_4 >> \mbf W,
\end{CD}
\]
where the morphism $w_1$ is defined by
$w_1 ( X^1, X^4, X^3,\sigma_1,\sigma_2) = (\check X^1, \check X^4, \check X^3, \check \sigma_1, \check \sigma_2, \mathcal V_1,\mathcal  V_4, \mathcal V_3)$
with $\mathcal V_4 =\mbox{im} \; (q_i^4+\sum_{h\in \Omega: h'=i} x_h^4)$  and
$\mathcal V_a=\mbox{im} \; (q_i^a+\sum_{h\in \Omega: h'=i} \sigma_{a, h''}^{-1}x_h^a)$ for any $ a=1, 3$.

Let $i_2: \mbf Y_2^c\hookrightarrow \mbf Y_2\hookleftarrow \mbf Y_2^o: j_2 $ be the inclusions. Then we see that $t_3i_2$ is a projective bundle of relative dimension $\nu_i-1$.
An argument similar to the proof of (\ref{E-i-i-LHS}) shows that the restriction of the complex $\FF_{\mu +\alpha_i-\alpha_j,\mu +\alpha_i} \EE_{\mu+\alpha_i,\mu}$
to $\mbf E^s_{\Omega}$ is equal to 
\begin{equation}
\label{E-i-i-RHS}
T_{2!}T_{1!} (\Q_{l,\mbf Z_2^s}) = S_3^* T_{4!} T_{3!} (\Q_{l,\mbf Y_2})=
S_3^* T_{4!}  J^o_{2!*} T^o_{3!}  (\Q_{l, \mbf Y_2^o}) \oplus  \oplus_{p=0}^{\nu_i -1} S_3^* T_{4!}\Q_{l, \mbf X_2^c}[m-2p],
\end{equation}
where $j^o_2$ and $t_3^o$ are the morphisms $\mbf Y_2^o\overset{t_3^o} {\to} \mbf X_2^o  \overset{j_2^o}{\to} \mbf X_2$ with $\mbf X_2^o$.
Finally, observe that there are isomorphisms $\mbf Y_1^o\simeq \mbf Y_2^o$, $\mbf X^c_1\simeq \mbf X^c_2$ and moreover the complex
$S_3^* T_{4!}(\Q_{l, \mbf X_2^c})$ is the restriction of $\II_{\mu}$ to $\mbf E^s_{\Omega}$. 
The lemma follows by comparing (\ref{E-i-i-LHS}) and (\ref{E-i-i-RHS}) and using 
the observations.
\end{proof}

\begin{lem} 
\label{Serre}
For any $i\neq j\in I$,  let $m=1-i\cdot j$. We have
\begin{equation*}
\begin{split}
&\bigoplus_{\substack{0\leq p\leq m \\ p \; even }}  
\EE^{(m-p)}_{\mu^3, \mu^2} 
\EE_{\mu^2, \mu^1} 
\EE^{(p)}_{\mu^1, \mu} 
=\bigoplus_{\substack{0\leq p\leq m\\ p \; odd} }  
\EE^{(m-p)}_{\mu^3,  \mu^2} 
\EE_{\mu^2, \mu^1} 
\EE^{(p)}_{\mu^1, \mu};\\
&\bigoplus_{\substack{0\leq p\leq m\\p \; even}  }  
\FF^{(p)}_{\mu,  \mu^1} 
\FF_{\mu^1,  \mu^2} 
\FF^{(m-p)}_{\mu^2,  \mu^3} 
=
\bigoplus_{\substack{0\leq p\leq m\\ p \; odd}  }  
\FF^{(p)}_{\mu,  \mu^1} 
\FF_{\mu^1,  \mu^2} 
\FF^{(m-p)}_{\mu^2,  \mu^3};
\end{split}
\end{equation*}
where 
$\mu^1=\mu+p\alpha_i$, 
$\mu^2= \mu+p\alpha_i+\alpha_j$, and
$\mu^3=\mu+m\alpha_i+\alpha_j$.
\end{lem}

\begin{proof}
Without lost of generality, we assume that $i$ is a source in $\Omega$. For $a=1, 2, 3, 4$, let $V^a$  be the $I$-graded vector spaces such that
\[
\dim V^1 =\nu-mi-j, \quad 
\dim V^2 =\nu -p i-j,\quad
\dim V^3=\nu-pi \quad \mbox{and} \quad
\dim V^4= \nu.
\]
Let  $\mbf Z$  be the variety of the data $(X^1 \overset{\rho_1}{\hookrightarrow} X^2\overset{\rho_2}{\hookrightarrow} X^3 \overset{\rho_3}{\hookrightarrow} X^4)$ where $X^a\in \mbf E_{\Omega}(D, V^a)$ for $a=1, 2, 3, 4$.  
Let $\pi: \mbf Z\to  \mbf E_{\Omega}(D, V^1,V^4)$ be the obvious projection.
Then 
\[
\EE^{(m-p)}_{\mu^3, \mu^2} 
\EE_{\mu^2, \mu^1} 
\EE^{(p)}_{\mu^1, \mu} 
=\Pi_! (\Q_{l, \mbf Z})[s_{m-p}],
\]
where 
$s_{m-p}=m(\nu(i)-\nu_i) +(d_j +\sum_{h\in \Omega: h'=j}\nu_{h''} -\nu_j)+(m-p) ( 1-(m-p))$.
Moreover, $\pi$ factors through $\mbf Z_{\Omega}(D, V^1, V^4)$, where the map $r$ from $\mbf Z$ to $\mbf Z_{\Omega}(D, V^1, V^4)$ is given by
$r (X^1 \overset{\rho_1}{\hookrightarrow} X^2\overset{\rho_2}{\hookrightarrow} X^3 \overset{\rho_3}{\hookrightarrow} X^4)=(X^1\overset{\rho_3\rho_2\rho_1}{\hookrightarrow} X^4)$.
Let $B_{m-p}=R_! (\Q_{l, \mbf Z})$.
Thus,
\[
\EE^{(m-p)}_{\mu^3, \mu^2} 
\EE_{\mu^2, \mu^1} 
\EE^{(p)}_{\mu^1, \mu} 
=\Pi_{!} (\Q_{l, \mbf Z_p})[s_{m-p}]
=\Pi_{12!} B_{m-p}[s_{m-p}].
\]
The identity for the $\EE$'s  is reduced to show that 
\begin{equation}
\label{even-odd}
\bigoplus_{\substack{0\leq p\leq m\\ p\; even}} B_{m-p}[(m-p) ( 1-(m-p))] =\bigoplus_{\substack{0\leq p\leq m \\ p \; odd}} B_{m-p}[(m-p)(1-(m-p))].
\end{equation}
This is shown in ~\cite[2.5.8]{Zheng08}.
For the sake of completeness, let us reproduce here.
Let $\mbf Z^s$ be the open subvariety of $\mbf Z$ defined by the condition that $X^a(i)$ are injective for $a=1, 2, 3, 4$.  The variety $\mbf Z^s_{\Omega}(D, V^1, V^4)$ is defined similarly.  Let $s: \mbf Z^s\to \mbf Z^s_{\Omega}(D, V^1, V^4)$ be the restriction of $r$ to $\mbf Z^s$ and 
$C_{m-p}:= S_!(\Q_{l,\mbf Z^s})$.
 Then the condition on the localization implies that we only need to show the
identity (\ref{even-odd}) when we restrict the complexes involved  to the variety $\mbf Z^s_{\Omega}(D, V^1, V^4)$, 
i.e., to show that  (\ref{even-odd}) holds with the complexes $B_{m-p}$ replaced by 
the complexes $C_{m-p}$.

Observe that the group $G_{V^2}\times G_{V^3}\times \mrm{GL}(V^1_i)\times \mrm{GL}(V^4_i)$ acts freely on $\mbf Z^s$ and the quotient  variety $\mbf Y$ is the variety 
of tuples $(\check X^1\overset{\rho}{\hookrightarrow} \check X^4, \mathcal V_1, \mathcal V_2, \mathcal V_3)$ 
where $\mathcal V_1,\mathcal V_2 \subseteq V^1(i)$, $ \mathcal V_3\subseteq V^4(i))$ such that 
$\mathcal V_1\subseteq V_2$ and $\rho(\mathcal V_2)\subseteq \mathcal V_3$; and $\dim \mathcal V_1=\nu_i -m$, $\dim \mathcal V_2=\nu_i -p$ and $\mathcal V_3=\nu_i$.
Moreover, the group $\mrm{GL}(V^1_i)\times \mrm{GL}(V^4_i)$ acts freely on $\mbf Z^s_{\Omega}(D, V^1, V^4)$ and its quotient variety $\mbf X$ 
is the variety obtained from $\mbf Y$ 
by deleting $\mathcal V_2$ and replacing the condition $\rho(\mathcal V_2)\subseteq \mathcal V_3$ by $\rho(\mathcal V_1)\subseteq \mathcal V_3$.
Let $t: \mbf Y\to \mbf X$ be the projection. Let $A_{m-p}=T_!(\Q_{l, Y})$. To show (\ref{even-odd}), it reduces to show that the identity holds 
with the complexes $B_{m-p}$ replaced by 
the complexes $A_{m-p}$.

Now define a partition $(\mbf X_n)_{n=0}^m$ of $\mbf X$ such that elements in $\mbf X_n$ satisfying  the condition that 
$\dim \rho(V^1(i))\cap \mathcal V_3= \nu_i -m+n$.  Let $\mbf Y_n= t^{-1}(\mbf X_n)$, and $t_n:\mbf  Y_n\to\mbf  X_n$ be the restriction of $t$ to $\mbf Y_n$. 
Then the restriction of $r$ to $\mbf Y_n$ has fiber at any point of $\mbf X_n$
isomorphic to the Grassmannian $\mrm{Gr}(m-p, n)$ of $(m-p)$-subspaces in  $n$-space.
By the property of the  cohomology of $\mrm{Gr}(m-p, n)$, we have  
\[
T_{n!} (\Q_{l,\mbf Y_n}) = \oplus_{\kappa} \Q_{l,\mbf X_n} [-2\sum_{a=1}^{m-p} (\kappa_a-a)],
\] 
where 
$\kappa$ runs through the sequences $(1\leq \kappa_1< \kappa_2<\cdots < \kappa_{m-p} \leq n)$.
Since the complexes $A_{m-p}$ are semisimple, it suffices to show that  (\ref{even-odd})  holds when restricts to the strata $X_n$ for all $n$, which is left to show that 
\[
\bigoplus_{\substack{ 0\leq p\leq n\\ p\; even}} \oplus_{\kappa} \Q_{l,\mbf X_n} [-2\sum_{a=1}^{p} (\kappa_a-a)+p(1-p)]=
\bigoplus_{\substack{0\leq p\leq n \\ p\; odd}} \oplus_{\kappa} \Q_{l,\mbf X_n} [-2\sum_{a=1}^{p} (\kappa_a-a)+p(1-p)].
\]
To any sequence $\kappa=(1\leq \kappa_1< \cdots <\kappa_p\leq n)$ of odd length, attached a sequence $\kappa'$ of even length  by 
$\kappa'_a=\kappa_{a+1} $ for $a=1,\cdots, p$,  if $\kappa_1=1$;
$\kappa'_1=1$ and  $\kappa'_{a+1}=\kappa_a$ for $a=1,\cdots, p$ if $\kappa_1\neq 1$.
This defines a  bijection between the set of sequences $\kappa$ of even length and the set of sequences $\kappa$ of odd length and it is clear that the shifts on both sides
are the same under this bijection. Thus the identity holds.
The identity for the $\FF$'s can be proved similarly.
\end{proof}

\begin{lem} For any $i\in I$ and $m\in \mbb N$,  we have
\label{EE=E}
\begin{equation*}
\begin{split}
&\EE_{\mu+(m+1)\alpha_i, \mu+m\alpha_i} \EE^{(m)}_{\mu+m \alpha_i, \mu} = \bigoplus_{0\leq p\leq m} \EE^{(m+1)}_{\mu+(m+1)\alpha_i, \mu} [m-2p];\\
&\FF_{\mu-(m+1)\alpha_i, \mu-m\alpha_i} \FF^{(m)}_{\mu-m \alpha_i, \mu} = \bigoplus_{0\leq p\leq m} \FF^{(m+1)}_{\mu-(m+1)\alpha_i, \mu} [m-2p].
\end{split}
\end{equation*}
\end{lem}

\begin{proof}
Fix three $I$-graded vector spaces $V^a$ for $a=1, 2,3$ such that $\dim V^1 =\nu-(m+1)i$, $\dim V^2=\nu-mi$ and $\dim V^3=\nu$. Let
$\mbf Z_1=\mbf Z_{\Omega}(D, V^1, V^3)$ and  $\mbf Z$ be the variety of tuples $(X^1, X^2, X^3, \rho_1, \rho_2)$, where $(X^1, X^2, X^3)\in \mbf E_{\Omega}(D, V^1, V^2, V^3)$, such that $X^1\overset{\rho_1}{\hookrightarrow} X^2 \overset{\rho_2}{\hookrightarrow} X^3$.
Let $t: \mbf Z \to \mbf Z_1$ be the map defined by $t(X^1, X^2, X^3, \rho_1, \rho_2)=(X^1, X^3, \rho_2\rho_1: X^1\to X^3)$. As in the proof of Lemma \ref{E-F-i}, we have 
$T_! (\Q_{l,\mbf Z} )=\oplus_{p=0}^{m} (\Q_{l,\mbf Z_1})[-2p]$.  So
\begin{equation*}
\begin{split}
\EE&_{\mu+(m+1)\alpha_i, \mu+m\alpha_i}  \EE^{(m)}_{\mu+m \alpha_i, \mu}=\Pi_{12!} T_! (\Q_{l,\mbf Z}) [e_{\mu+(m+1)\alpha_i, \alpha_i}+e_{\mu+m\alpha_i, m \alpha_i} ]\\
&=  \oplus_{p=0}^m    \Pi_{12!}  (\Q_{l,\mbf Z_1})[e_{\mu+(m+1)\alpha_i, \alpha_i}+e_{\mu+m\alpha_i, m \alpha_i}-2p]=
\oplus_{p=0}^m \EE^{(m+1)}_{\mu+(m+1)\alpha_i, \mu}  [m-2p].
\end{split}
\end{equation*}
The proof for the $\FF$'s is similar.
\end{proof}

We refer to ~\cite[I]{Li10b} and ~\cite[Ch. 23]{Lusztig93} 
for the definitions of the quantum modified algebra $\dot{\U}$ and its integral form $_{\mbb A} \! \dot{\U}$ associated to the graph $\Gamma$.
By specializing the shift $[z]$ to $v^z$ for any $z\in \mbb Z$, the identities  in Lemmas ~\ref{I-I}-\ref{EE=E} 
become  the defining relations of the integral form $_{\mbb A}\! \dot{\U}$.
In short, we have

\begin{thm}
\label{U-relations}
The complexes  $\II_{\mu}$, $\EE^{(n)}_{\mu, \mu-n \alpha_i}$ and $\FF^{(n)}_{\mu, \mu +n \alpha_i}$ 
satisfy the defining relations of the integral form $_{\mbb A}\! \dot{\U}$.
\end{thm}

\begin{rem}
In many, if not all, respects, the proof of Theorem ~\ref{U-relations} is very similar to that of ~\cite[Theorem  2.5.2]{Zheng08}
(see also Proposition ~\ref{Psi-relation} in this paper).
\end{rem}

\subsection{Complex $K_{\bullet}$}

Consider the complexes of the form 
\begin{equation}
\label{complex}
K_{\bullet} =K_1 \cdot K_2  \cdot  ... \cdot K_m \quad \in \mathscr D^-_{\G}(\mbf E_{\Omega}(D, V^1, V^2)),
\end{equation}
where the $K_a$'s are either $\EE^{(n)}_{\mu', \mu}$ or $\FF^{(n)}_{\mu', \mu}$.

\begin{prop}
\label{boundedness}
The complexes $K_{\bullet}$ in (\ref{complex}) are bounded.
\end{prop}

\begin{proof}
For  any pair $(\mbf{i, a})$ of sequences, where $\mbf i=(i_m,\cdots,i_1)\in I^m$ and $\mbf a=(a_m,\cdots, a_1)\in \mbb N^m$, we write
\[
\EE_{(\mbf{i, a}), \mu} =\EE^{(a_m)}_{\mu,\mu^{m-1}} \cdots \EE^{(a_2)}_{\mu^2,\mu^1} \EE^{(a_1)}_{\mu^1,\mu^0}
\quad \mbox{and}\quad 
\FF_{\mu, (\mbf{i, a})}=
\FF^{(a_1)}_{\mu^0,\mu^1} \cdots \FF^{(a_{m-1})}_{\mu^{m-2},\mu^{m-1}} \FF^{(a_m)}_{\mu^{m-1},\mu},
\]
such that 
$\mu^l -\mu^{l-1}= a_l \alpha_{i_l}$ for $l=1,\cdots, m$. By Lemmas ~\ref{E-i-j} and ~\ref{E-F-i}, 
it suffices to show the boundedness of the complex $K_{\bullet}$ if
$K_{\bullet}$ is of the form $\FF_{\mu, (\mbf{j, b})}\EE_{(\mbf{i,a}), \mu}$ for any two pairs $(\mbf{i, a})$ and $(\mbf{j, b})$.
An argument similar to the proof of Lemma ~\ref{Serre} yields that 
\[
\FF_{\mu, (\mbf{j, b})}\EE_{(\mbf{i,a}), \mu}=\Pi_! (\Q_{l, \mbf Z})[m], 
\]
for some $m$, where $\mbf Z$ is the variety of the data $( X^1 \overset{\rho_1}{\hookleftarrow} X^2 \overset{\rho_2}{\hookrightarrow} X^3) $ and $\pi$ is the projection from $\mbf Z$ to the variety
$\mbf E_{\Omega}(D, V^1, V^3)$ with the dimensions of $V^1 $ and $V^3$ determined by the pairs of sequences.
The morphism $\pi$ factors through the following varieties
\[
\mbf Z \overset{\pi_1}{\to} \mbf Z_1 \overset{\pi_2}{\to} \mbf Z_2 \overset{\pi_3}{\to} \mbf E_{\Omega} (D, V^1, V^3),
\]
where 
$\mbf Z_1$ is the variety obtained from $\mbf Z$ by forgetting the maps $\rho_2$, the variety $\mbf Z_2$ is the quotient variety of $\mbf Z_1$ by the group $G_{V^2}$
and the morphisms are clearly defined. 
It is clear that the functors $\Pi_{1!}$, $\Pi_{2!}$ and $\Pi_{3!}$ send bounded complexes to bounded complexes and  $\Pi_!=\Pi_{3!} \Pi_{2!} \Pi_{1!}$. The proposition follows. 
\end{proof}

\begin{rem}
Since $\pi_2$ is a quotient map and $\pi_3$ is proper in the above proof, 
the semisimplicity of the  complexes $K_{\bullet}$ is reduced to show that 
 $\Pi_{1!} (\Q_{l,\mbf Z})  $ is semisimple. This is again reduced to show that $\II_{\mu}$ is semisimple, 
 or more precisely, $\II_{\mu}=\widetilde{\IC} (\overline{\mbox{im} \; \pi_{12}})$. 
\end{rem}

\subsection{Functor $\T$}

Let  $\FF^-_{\Omega, \G}(D, V^1, V^2)$ be the category of functors from $\DD^-_{\G}(\mbf E_{\Omega}(D, V^1))$ to $\DD^-_{\G}(\mbf E_{\Omega}(D, V^2))$. 
Define a functor
\[
\Theta_{\Omega}: \DD^-_{\G}(\mbf E_{\Omega}(D, V^1, V^2)) \to \FF^-_{\Omega, \G}(D, V^1, V^2) 
\]
by  $\Theta_{\Omega}(K)= P_{2!} (K\otimes P_1^*(-))$ for any object $K$ in $  \DD^-_{\G}(\mbf E_{\Omega}(D, V^1, V^2))$ and the functors $P_{2!}$ and $P_1^*$ are defined 
in (\ref{Pi}).

\begin{prop}
\label{theta-convolution}
$\Theta_{\Omega}(K\cdot L) = \Theta_{\Omega}(L) \Theta_{\Omega}(K)$ for any objects $K$ in $\DD^-_{\G}(\mbf E_{\Omega}(D, V^1, V^2))$
and $L$  in $\DD^-_{\G}(\mbf E_{\Omega}(D, V^2, V^3))$.
\end{prop}

\begin{proof}
By definition, we have 
\[
\T(L)   \T(K) (M) 
= P'_{2!}(L\otimes (P'_1)^* \T(K) (M)) = P'_{2!}(L\otimes (P'_1)^*  P_{2!} (K\otimes P_1^*(M)),
\]
where $P'_{2!}$ and $(P'_1)^*$ are corresponding to the maps $p'_2$ and $p_1'$ in  the following cartesian diagram
\[
\begin{CD}
\mbf E_{\Omega}(D, V^1, V^3) @< p_{13}<< \mbf E_{\Omega}(D, V^1, V^2, V^3) @>p'_{12}>> \mbf E_{\Omega}(D, V^1, V^2)\\
@V \tilde p_2VV @Vp_{23}VV @Vp_2VV\\
\mbf E_{\Omega}(D, V^3) @<p'_2<< \mbf E_{\Omega}(D, V^2, V^3) @>p'_1>> \mbf E_{\Omega}(D, V^2).
\end{CD}
\]
By an argument similar to ~\cite[(16)]{Li10b}, we have 
$(P'_1)^*  P_{2!}= P_{23!}  (P'_{12})^*$. So
\[
\T(L)   \T(K) (M)=P'_{2!}(L\otimes P_{23!} (P'_{12})^*(K\otimes P_1^*(M)).
\]
By an argument similar to ~\cite[(19), (21)]{Li10b}, we have 
$P_{23!}( A\otimes P_{23}^*(B) )= P_{23!} (A) \otimes B$. Thus,
\begin{equation}
\label{T-LHS}
\begin{split}
\T(L)   \T(K) (M)
&=P'_{2!} P_{23!} ( P_{23}^* (L) \otimes (P'_{12})^*(K\otimes P_1^*(M)))\\
&=P'_{2!} P_{23!} ( P_{23}^* (L) \otimes (P_{12}')^*(K)\otimes (P_{12}')^* P_1^*(M)).
\end{split}
\end{equation}
Similarly, we have 
\begin{equation}
\label{T-RHS}
\begin{split}
\T(K\cdot L) 
= \tilde P_{2!} P_{13!} ((P_{12}')^*(K) \otimes P_{23}^*(L) \otimes P_{13}^* \tilde P_1^* (M)),
\end{split}
\end{equation}
where $\tilde P^*_1$ comes from the projection $\mbf E_{\Omega}(D, V^1, V^3)\to \mbf E_{\Omega}(D, V^1)$.
The lemma follows by comparing (\ref{T-LHS}) with (\ref{T-RHS}) and  the following identity
\[
P'_{2!} P_{23!} = \tilde P_{2!} P_{13!} \quad \mbox{and}\quad 
(P_{12}')^* P_1^*=P_{13}^* \tilde P_1^*,
\]
which can be proved by a similar way as ~\cite[(18), (20)]{Li10b}. 
\end{proof}

Define a functor of equivalence
\[
\Psi_{\Omega}^{\Omega'}: \FF^-_{\G, \Omega}(D, V^1, V^2) \to \FF^-_{\G, \Omega'}(D, V^1, V^2)
\]
by $\Psi_{\Omega}^{\Omega'} ( F) = \Phi_{\Omega}^{\Omega'} F a^* \Phi_{\Omega'}^{\Omega}$, where
$a$ is the map of multiplication by $-1$ along the fiber of the vector bundle $\mbf E_{\Omega} $ over $\mbf E_{\Omega\cap \Omega'}$ . 
Its inverse is given by $\Psi_{\Omega'}^{\Omega} (-) = a^* \Phi_{\Omega'}^{\Omega} (-) \Phi_{\Omega}^{\Omega'}$,
since $\Phi_{\Omega'}^{\Omega} \Phi_{\Omega}^{\Omega'} = a^*$. Moreover,  we have

\begin{lem}
$\Psi_{\Omega}^{\Omega'}$ commutes with the composition:
$\Psi_{\Omega}^{\Omega'} (F_2\circ F_1) = \Psi_{\Omega}^{\Omega'}(F_2) \circ \Psi_{\Omega}^{\Omega'}(F_1)$ for any
$F_1\in \FF^-_{\G,\Omega}(D, V^1, V^2) $ and $F_2\in \FF^-_{\G,\Omega}(D, V^2, V^3)$.
\end{lem}

Let $^1\DD^-_{\G}(\mbf E_{\Omega}(D, V^1,V^2))$ 
(resp. $^1\DD^-_{\G}(\mbf E_{\Omega}(D, V^i))$, $i=1, 2$) 
be the full subcategory of $\DD^-_{\G}(\mbf E_{\Omega}(D, V^1,V^2)) $ (resp. $\DD^-_{\G}(\mbf E_{\Omega}(D, V^i))$)
consisting of all objects such that $a^* (K)\simeq K$.
Let $^1 \FF^-_{\G, \Omega}(D, V^1, V^2)$ denote the category of functors from the category
$^1\DD^-_{\G}(\mbf E_{\Omega}(D, V^1))$ to $^1\DD^-_{\G}(\mbf E_{\Omega}(D, V^2))$. 
We have

\begin{lem}
\label{Theta-Phi-Psi}
The following diagram commutes
\[
\begin{CD}
^1\DD^-_{\G}(\mbf E_{\Omega}(D, V^1,V^2)) @>\Theta_{\Omega}>> ^1 \FF^-_{\G, \Omega}(D, V^1, V^2)\\
@V\Phi_{\Omega}^{\Omega'} VV @V\Psi_{\Omega}^{\Omega'} VV\\
^1\DD^b_{\G}(\mbf E_{\Omega'}(D, V^1,V^2)) @>\Theta_{\Omega'}>> ^1\FF_{\G, \Omega'}(D, V^1, V^2).
\end{CD}
\]
\end{lem}

\begin{proof}
For any $K\in\;  ^1\DD^-_{\G}(\mbf E_{\Omega}(D, V^1,V^2))$ and $K_1\in \; ^1\DD^-_{\G}(\mbf E_{\Omega}(D, V^1))$, we have 
\begin{equation*}
\begin{split}
\Psi_{\Omega}^{\Omega'} \T(K) (K_1) 
= \Phi_{\Omega}^{\Omega'} \T(K) a^* \Phi_{\Omega'}^{\Omega} (K_1) 
= \Phi_{\Omega}^{\Omega'} P_{2!} (K\otimes a^* P_1^*\Pi_{1!} ((\Pi'_1)^*(K_1)\otimes \mathcal L_1))[d_1],
\end{split}
\end{equation*}
where $P_{2!}$ and  $P_1^*$ are from (\ref{Pi}), $\Pi_{1!}$ and $(\Pi_1')^*$ come from the following projections
\[
\begin{CD}
\mbf E_{\Omega}(D, V^1) @<\pi_1<< \mbf E_{\Omega\cup \Omega'} (D, V^1) @>\pi_1'>> \mbf E_{\Omega'}(D, V^1);
\end{CD}
\]
$d_1$ is the rank of $\pi_1$ and $\mathcal L_1$ is defined in (\ref{L}).
Consider the following cartesian diagram
\[
\begin{CD}
\mbf E_{\Omega\cup \Omega'} (D, V^1)\times \mbf E_{\Omega}(D, V^2) @>\tilde \pi_1>> \mbf E_{\Omega}(D, V^1, V^2)\\
@V\tilde p_1VV @Vp_1 VV\\
\mbf E_{\Omega\cup \Omega'}(D, V^1) @>\pi_1>> \mbf E_{\Omega}(D, V^1).
\end{CD}
\]
By an argument similar to ~\cite[(18)]{Li10b}, we have $P_1^* \Pi_{1!} =\tilde \Pi_{1!} \tilde P_1^*$. So
\begin{equation*}
\begin{split}
\Psi&_{\Omega}^{\Omega'} \T(K) (K_1) 
=\Phi_{\Omega}^{\Omega'} P_{2!} (K\otimes a^*  \tilde \Pi_{1!} \tilde P_1^* ((\Pi'_1)^*(K_1)\otimes \mathcal L_1))[d_1]\\
&=\Phi_{\Omega}^{\Omega'} P_{2!} \tilde \Pi_{1!}(\tilde \Pi_1^* (K)\otimes a^*  \tilde P_1^* (\Pi'_1)^*(K_1)\otimes a^* \tilde P^*_1\mathcal L_1))[d_1]\\
&=R_{2!} (R_1^* P_{2!} \tilde \Pi_{1!} (\alpha) \otimes \mathcal L_2)[d_2],
\end{split}
\end{equation*}
where $R_{2!}$ and $R_1^*$ come from the following projections
\[
\begin{CD}
\mbf E_{\Omega}(D, V^2) @<r_1<< \mbf E_{\Omega\cup \Omega'}(D, V^2) @>r_2>> \mbf E_{\Omega'}(D, V^2),
\end{CD}
\]
$d_2$ is the rank of $r_1$,  $\alpha =\tilde \Pi_1^* (K)\otimes a^*  \tilde P_1^* (\pi'_1)^*(K_1)\otimes a^* \tilde P^*_1\mathcal L_1)[d_1]$, and 
$\mathcal L_2$ is defined in (\ref{L}). 
The following cartesian diagram
\[
\begin{CD}
\mbf E_{\Omega\cup \Omega'} (D, V^1, V^2) @>t_1>> \mbf E_{\Omega\cup \Omega'}(D, V^1)\times \mbf E_{\Omega}(D, V^2)\\
@Vs_1VV @Vp_2\tilde \pi_1VV\\
\mbf E_{\Omega\cup \Omega'}(D, V^2) @>r_1>> \mbf E_{\Omega}(D, V^2),
\end{CD}
\]
gives rise to the identity
$R_1^* P_{2!}\tilde \Pi_{1!} =S_{1!}T_1^*$. So we have 
\begin{equation}
\label{Psi-LHS}
\begin{split}
\Psi&_{\Omega}^{\Omega'} \T(K) (K_1) =R_{2!}(S_{1!}T_1^*(\alpha)\otimes \mathcal L_2)[d_2]
=R_{2!} S_{1!} (T_1^* (\alpha) \otimes S_1^* \mathcal L_2)[d_2]\\
&=R_{2!} S_{1!} (T_1^*\tilde \Pi_1^* (K)\otimes T_1^*  \tilde P_1^* (\Pi'_1)^* a^* (K_1)\otimes a^* T_1^*  \tilde P^*_1\mathcal L_1\otimes S_1^* \mathcal L_2)[d_1+d_2].
\end{split}
\end{equation}
On the other hand, we have
\begin{equation}
\label{Psi-RHS}
\begin{split}
\Theta_{\Omega'} \Phi_{\Omega}^{\Omega'}(K) (K_1) 
&=P_{2!}' (\Phi_{\Omega}^{\Omega'} (K)\otimes (P_1')^*(K_1))\\
&=P_{2!}' (M_{12!}' (M_{12}^*(K)\otimes \mathcal L_{12})[r_{12}]\otimes (P_1')^*(K_1))\\
&=P_{2!}' M'_{12!} (M_{12}^*(K) \otimes \mathcal L_{12} \otimes (M_{12}')^* (P_1')^* (K_1) ) [r_{12}],
\end{split}
\end{equation}
where $P_{2!}'$, $(P'_1)^*$ come from the following projections
\[
\begin{CD}
\mbf E_{\Omega'}(D, V^1) @<p_1'<< \mbf E_{\Omega'}(D, V^1, V^2) @>p_2'>> \mbf E_{\Omega'}(D, V^2),
\end{CD}
\]
and $M_{12!}'$, $M_{12}^*$, $\mathcal L_{12}$  and  $r_{12}$  are from \ref{Fourier}.
By comparing (\ref{Psi-LHS}) with (\ref{Psi-RHS}), the lemma follows from the following observations:
$r_2 s_1=p_2' m_{12}'$, $\tilde \pi_1 t_1=m_{12}$, $\pi_1' \tilde p_1 t_1=p_1'm_{12}'$,  $p_1t_1=p_1'$, $s_1=p_2'$ and 
$\mathcal L_{12} =a^* (P_1')^* \mathcal L_1\otimes (P_2')^* \mathcal L_2$.
Note that the last identity can be deduced  from the following  well-known fact. Let $s, p_1, p_2: k\times k\to k$ be the addition, first and second projections, respectively. Then
$s^* \mathcal L_{\chi} = p_1^* \mathcal L_{\chi}\otimes p_2^*\mathcal L_{\chi}$.
\end{proof}

We define the following functors in $\FF^-_{\Omega,\G}(D, V^1, V^2)$:
\begin{align*}
\mathfrak{ I}_{\mu}  &= \Pi_{2!}\Pi_1^*,  &&\mbox{if}\; \dim V^1=\dim V^2=\nu; \\
\mathfrak {F}^{(n)}_{\mu, \mu-n \alpha_i}  &= \Pi_{2!} \Pi_1^* [e_{\mu,n\alpha_i}]  , \; &&\mbox{if}\;  \dim V^1=\nu \;  \mbox{and}\;  \dim V^2=\nu+ ni;\\ 
 \mathfrak {E}^{(n)}_{\mu, \mu +n \alpha_i} &=  \Pi_{1!} \Pi_2^* [f_{\mu,n\alpha_i}] , &&\mbox{if} \; \dim V^1=\nu \;\mbox{and}\; \dim V^2=\nu-ni;
\end{align*}
where the functors $\Pi_{i!}$ and $\Pi_i^*$ are defined in (\ref{Pi}) and $e_{\mu,n\alpha_i}$ and $f_{\mu,n\alpha_i} $ are defined in (\ref{coefficient}).
Note that $\mathfrak I_{\mu} =\mbox{Id}_{\DD^-_{\G}(\mbf E_{\Omega}(D, V^1))}$, the identity functor,  since $\pi_1$ and $\pi_2$ are principal  $\G_{V^1}$-bundles.
We have 

\begin{prop}
\label{Theta-generator}
$\Theta_{\Omega}(\II_{\mu})=\mathfrak{ I}_{\mu}$, 
$\Theta_{\Omega}(\EE^{(n)}_{\mu, \mu-n \alpha_i})  = \mathfrak {F}^{(n)}_{\mu, \mu-n \alpha_i}  $ and
$ \Theta_{\Omega}(\FF^{(n)}_{\mu, \mu +n \alpha_i})=\mathfrak {E}^{(n)}_{\mu, \mu +n \alpha_i} $.
\end{prop}

\begin{proof}
We shall show that $\Theta_{\Omega}(\EE^{(n)}_{\mu, \mu-n \alpha_i})  = \mathfrak {F}^{(n)}_{\mu, \mu-n \alpha_i}  $.
For any $K_1\in \DD^-_{\G}(\mbf E_{\Omega}(D, V^1))$, we have
\begin{equation*}
\begin{split}
\mathfrak F^{(n)}_{\mu,\mu-n\alpha_i} (K_1) &=\Pi_{2!}\Pi_1^*(K_1) [e_{\mu, n\alpha_i}]=
P_{2!} \Pi_{12!} \Pi_{12}^* P_1^* (K_1) [e_{\mu, n\alpha_i}]\\
&=P_{2!} \Pi_{12!} (\bar{\mbb Q}_{l,\mbf Z_{\Omega}}\otimes \Pi_{12}^* P_1^* (K_1) [e_{\mu, n\alpha_i}]
=P_{2!} ( \Pi_{12!} (\bar{\mbb Q}_{l,\mbf Z_{\Omega}})[e_{\mu, n\alpha_i}] \otimes P_1^* (K_1))\\
&=P_{2!} ( Q( \pi_{12!} (\bar{\mbb Q}_{l,\mbf Z_{\Omega}})[e_{\mu, n\alpha_i}] )\otimes P_1^* (K_1))
=\Theta_{\Omega}(\EE^{(n)}_{\mu, \mu-n \alpha_i}) (K_1).
\end{split}
\end{equation*}
The rest can be proved similarly.
\end{proof}

By Lemmas ~\ref{generator-Fourier},  ~\ref{Theta-Phi-Psi}, and Proposition ~\ref{Theta-generator},  we have

\begin{cor}
\label{Psi-generator}
$\Psi_{\Omega}^{\Omega'} (\mathfrak I_{\mu})  =\mathfrak I_{\mu}$, 
$\Psi_{\Omega}^{\Omega'} (\mathfrak {F}^{(n)}_{\mu, \mu-n \alpha_i})=\mathfrak {F}^{(n)}_{\mu, \mu-n \alpha_i}$  and 
$\Psi_{\Omega}^{\Omega'} (\mathfrak {E}^{(n)}_{\mu, \mu+n \alpha_i})=\mathfrak {E}^{(n)}_{\mu, \mu+n \alpha_i}$.
\end{cor}

Actually, we need to show that the complexes in Corollary ~\ref{Psi-generator} are invariant under the functor $a^*$. This can be proved as in ~\cite[10.2.4]{Lusztig93}.

From Proposition ~\ref{theta-convolution}, Corollary ~\ref{Psi-generator} and Theorem ~\ref{U-relations}, we have 

\begin{prop}
\label{Psi-relation}
The functors $\mathfrak I_{\mu}$, $\mathfrak E^{(n)}_{\mu, \mu-n \alpha_i}$ and $\mathfrak F^{(n)}_{\mu, \mu +n \alpha_i}$ 
satisfy the defining relations of $_{\mbb A}\! \dot{\U}$.
\end{prop}

From Corollary ~\ref{Psi-generator}, one sees  that the functors $\mathfrak {F}^{(n)}_{\mu, \mu-n \alpha_i} $ and 
$\mathfrak {E}^{(n)}_{\mu, \mu +n \alpha_i} $ are the functors $\mathfrak F^{(n)}_{\nu, i}$ and $\mathfrak E^{(n)}_{\nu, i}$ in ~\cite{Zheng08}, respectively. 
Proposition ~\ref{Psi-relation} was first proved in ~\cite[2.5.8]{Zheng08}.

Now that the functor $\Theta_{\Omega}$ induces a bifunctor 
\begin{equation}
\label{action}
\circ: \DD^-_{\G}(\mbf E_{\Omega}(D, V^1, V^2) )\times\DD^-_{\G}(\mbf E_{\Omega}(D, V^1) ) \to \DD^-_{\G}(\mbf E_{\Omega}(D, V^2) )
\end{equation}
given by $K\circ K_1= \Theta_{\Omega} (K) (K_1) = P_{2!} (K\otimes P_1^* (K_1))$ 
for any $K\in \DD^-_{\G}(\mbf E_{\Omega}(D, V^1, V^2) )$ and $K_1\in \DD^-_{\G}(\mbf E_{\Omega}(D, V^1) )$.

Suppose that the complexes $K_{\bullet}$ are semisimple, then we may form an associative algebra  over the ring $\mbb Z[v,v^{-1}]$ of Laurent polynomials.
\[
\KK_d=\bigoplus_{\nu^1,\nu^2\in \mbb N[I]} \KK_{d,\nu^1,\nu^2},
\]
where $\KK_{d,\nu^1,\nu^2}$ is the free $\mbb Z[v,v^{-1}]$-module spanned by the isomorphism classes of simple perverse sheaves
appearing in $K_{\bullet}$ in $\DD^-_{\G}(\mbf E_{\Omega}(D, V^1, V^2) )$.
The multiplication on $\KK_d$ is descended from the convolution product ``$\cdot$'' in (\ref{cdot}).

Let $V_{\underline \lambda}=V_{\lambda_1}\otimes\cdots\otimes V_{\lambda_n}$ be the tensor product of the irreducible integrable representations of $\dot{\U}$ 
with highest weights $\lambda_1$, $\cdots$,  $\lambda_n$ in $\X^+$. Denote by $\DD_{\underline \lambda}$  the full subcategory of 
$\oplus_{V} \DD^-_{\G}(\mbf E_{\Omega}(D, V) )$ such that its  Grothendieck group $\mathscr V_{\underline \lambda}$  
is isomorphic to the integral form of  $V_{\underline \lambda}$ (see ~\cite{Zheng08}).
Let $\QQ_d$  be the full subcategory of $\oplus_{V^1, V^2} \DD^-_{\G}(\mbf E_{\Omega}(D, V^1, V^2))$ consisting of 
all semisimple complexes whose simple summands are from $K_{\bullet}$, up to shifts.
Then the bifunctor (\ref{action}) gives rise to a bifunctor
$\QQ_d \times \DD_{\underline \lambda} \to \DD_{\underline \lambda}$  by restriction, which descends to a bilinear map
\[
\circ : \KK_d \times \mathscr V_{\underline \lambda} \to \mathscr V_{\underline \lambda}.
\]
Let $\BB_d$ (resp. $\BB_{\underline \lambda}$) be the set of all isomorphism classes of simple perverse sheaves appearing in $\QQ_d$ (resp. $\DD_{\underline \lambda}$).
We then have
\[
a\circ b =\sum_{c\in \BB_{\underline \lambda}} s_{a, b}^c c, \quad \mbox{where} \; s_{a, b}^c\in \mbb N[v, v^{-1}],
\]
for any $a\in \BB_d$ and $b\in \BB_{\underline \lambda}$. From this, we have 

\begin{cor}
If the complexes $\K_{\bullet}$ are semisimple and Conjecture 4.14 in ~\cite{Li10b} holds, then the action of the canonical basis elements in $\dot{\U}$ 
on the canonical basis elements in $V_{\underline \lambda}$ has structure constants  in $\mbb N[v,v^{-1}]$ with respect to the canonical basis in $V_{\underline \lambda}$.
\end{cor}

\begin{rem}
(1). It should be true that the functor $\Theta_{\Omega}$ is fully faithful.

(2). We are not sure if the superscript $1$ in the categories in Lemma ~\ref{Theta-Phi-Psi} can be dropped.

(3). The algebra $\KK_d$ should be the generalized $q$-Schur algebras (\cite{D03}) when the graph $\Gamma$ is of finite type.
\end{rem}

\end{document}